\documentclass[10pt,english,letterpaper,reqno]{amsart}

\title[Koszul duality for generalised Steinberg representations]{Koszul duality for generalised Steinberg representations of $p$-adic groups}

\author[C. Cunningham]{Clifton Cunningham}
\address{Department of Mathematics and Statistics, University of Calgary, 
2500 University Drive NW, 
Calgary, Alberta, 
T2N 1N4, 
Canada}
\email{ccunning@ucalgary.ca}
\thanks{Clifton Cunningham's research is supported by NSERC Discovery Grant RGPIN-2020-05220.}

\author[J. Steele]{James Steele}
\address{Department of Mathematics and Statistics, University of Calgary, 
2500 University Drive NW, 
Calgary, Alberta, 
T2N 1N4, 
Canada}
\email{james.steele@ucalgary.ca}
\thanks{Both authors are grateful to the \href{https://www.slmath.org}{Simons Laufer Mathematical Sciences Institute} for supporting the \href{https://www.slmath.org/summer-schools/1072}{Koszul Duality in the Local Langlands Program} summer school in 2024, 
where some of this paper was written.}

\usepackage[T1]{fontenc}

\usepackage{amssymb,url,xspace,amsthm,adjustbox}%, stix}

\usepackage[alphabetic,nobysame]{amsrefs}

\usepackage{mathtools}
\usepackage{graphicx}

\newcommand{\ZZ}{{\mathbb{Z}}}

\newcommand{\CC}{{\mathbb{C}}}

\DeclareMathOperator{\GL}{GL}
\DeclareMathOperator{\Sym}{Sym}
\DeclareMathOperator{\SL}{SL}

\newcommand{\op}{_{\operatorname{op}}}

\newcommand{\Lgroup}[1]{{\hskip-2 pt \,^L\hskip-1pt{#1}}}
\newcommand{\dualgroup}[1]{{\widehat{#1}}}

\newcommand{\Frob}{{\operatorname{Fr}}}

\DeclareMathOperator{\Ext}{Ext}

\newcommand{\Spec}[1]{{\operatorname{Spec}(#1)}}
\DeclareMathOperator{\Ad}{Ad}

\DeclareMathOperator{\End}{End}

\newcommand{\abs}[1]{{\vert #1 \vert}}
\newcommand{\ceq}{{\, :=\, }}

\newcommand{\iso}{{\ \cong\ }}

\DeclareMathOperator{\diag}{diag}

\newcommand{\Perv}{\mathbf{Per}}

\newcommand{\Rep}{\mathbf{Rep}}

\newcommand{\IC}{\text{IC}}

\newcommand{\Ft}{\operatorname{\mathsf{F\hskip-1pt t}}}

\makeatletter
\newcommand{\labitem}[2]{
\def\@itemlabel{\textbf{#1}}
\item
\def\@currentlabel{#1}\label{#2}}
\makeatother
\newcommand{\1}{{\mathbbm{1}}}

\newcommand{\Ind}{\operatorname{Ind}}

\newcommand{\C}{\mathbb{C}}

\newcommand{\Hom}{\mathrm{Hom}}

\newcommand{\St}{\mathrm{St}}

\newcommand{\bpm}{\begin{pmatrix}}
\newcommand{\epm}{\end{pmatrix}}

\newcommand{\bspm}{\left(\begin{smallmatrix}}
\newcommand{\espm}{\end{smallmatrix}\right)}
\newcommand{\Mod}{{\operatorname{\mathbf{Mod}}}}

\renewcommand{\St}{{\operatorname{St}}}
\newcommand{\Gmult}{\mathbb{G}_{\hskip-1.5pt\text{m}}}
\renewcommand{\IC}{\mathcal{I\hskip-1pt C}}
\newcommand{\AZ}{{\text{A\hskip-1.5ptZ}}}
\renewcommand{\Rep}{\operatorname{\mathbf{Rep}}}

\usepackage{float}
\usepackage{quiver}

\usepackage{datetime}
\usepackage[T1]{fontenc}
\usepackage[indent]{parskip}

\usepackage{amssymb}
\usepackage{mathrsfs,stmaryrd}
\usepackage{yfonts, bbm}
\usepackage{enumitem}
\usepackage{hyperref}
\hypersetup{
  colorlinks   = true, %Colours links instead of ugly boxes
  urlcolor     = blue, %Colour for external hyperlinks
  linkcolor    = blue, %Colour of internal links
  citecolor   = green %Colour of citations
}

\usepackage{tikz}
\usetikzlibrary{shapes,arrows,calc,matrix}
\usepackage{tikz-cd}
\usepackage{todonotes}
\usepackage{color}

\usepackage{amsmath}
\usepackage{lipsum}
\usepackage{setspace}

\theoremstyle{definition}
  
\newtheorem{theorem}{Theorem}[section]
\newtheorem{corollary}[theorem]{Corollary}
\newtheorem{lemma}[theorem]{Lemma}
\newtheorem{proposition}[theorem]{Proposition}

\newtheorem{definition}[theorem]{Definition}
\newtheorem{remark}[theorem]{Remark}
\newtheorem{example}[theorem]{Example}

\date{\today}                                
\begin{document}

%\dedicatory{}

\begin{abstract}
Let $G$ be a semisimple group, split over a non-Archimedean field $F$. We prove that the category of modules over the extension algebra of generalised Steinberg representations of $G(F)$ is equivalent to a full subcategory of equivariant perverse sheaves on the variety of Langlands parameters for these representations. Specifically, we establish an equivalence 
\[
\Mod(\Ext_G^\bullet(\Sigma_\lambda, \Sigma_\lambda))  \simeq \Perv_{\dualgroup{G}}^\circ(X_\lambda), 
\]
where $\Sigma_\lambda$ is the direct sum of generalised Steinberg representations and $\Perv_{\dualgroup{G}}^\circ(X_\lambda)$ is the subcategory of perverse sheaves on the variety of Langlands parameters $X_\lambda$  corresponding to these representations under Vogan's geometrisation of the Langlands correspondence. Furthermore, we demonstrate that this equivalence is a true Koszul duality by showing that the extension algebra of generalised Steinberg representations is Koszul dual to the endomorphism algebra of the direct sum of corresponding equivariant perverse sheaves, taken in the equivariant derived category $D_{\dualgroup{G}}^b(X_\lambda)$.
\end{abstract}
 
\maketitle
\tableofcontents

\setcounter{section}{-1}

\section{Introduction} 
In this paper we consider two categories that are built from the generalised Steinberg representations of $G(F)$, where $G$ is a semisimple group, split over a $p$-adic field $F$. Generalised Steinberg representations are exactly the irreducible admissible representations $\pi$ of $G(F)$ for which $H^i(G, \pi)$ is non-trivial, for some $i \geq 0$, described in detail in Section~\ref{section:spectral cat}. One of the two categories in our first main result, Theorem~\ref{thm:main1}, is modules over the extension algebra 
\[
\Ext_G^\bullet(\Sigma_\lambda, \Sigma_\lambda) = \bigoplus_{I,J\subseteq S} \Ext_{G}^{\delta(I,J)}(\sigma_I, \sigma_J),
\]
where $\Sigma_\lambda$ is the direct sum of the generalised Steinberg representations of $G(F)$, individually, written as $\sigma_I$ and parametrised by subsets of simple roots $I \subseteq S$ of $G$ when taken together. Here, extensions are taken in $\textbf{Rep}(G)$, the category of smooth representations of $G(F)$.
To describe the other category, we use Vogan's conception of the Langlands correspondence, matching generalised Steinberg representations with certain simple $\dualgroup{G}$-equivariant perverse sheaves on the variety $X_\lambda$ of Langlands parameters for these representations and we identify a full subcategory 
\[
\Perv_{\dualgroup{G}}^\circ(X_\lambda) \hookrightarrow \Perv_{\dualgroup{G}}(X_\lambda),
\]
called the principal block of $\Perv_{\dualgroup{G}}(X_\lambda)$ and defined precisely in Section~\ref{subsection:principal block}. In this paper we show that this Abelian subcategory is equivalent to modules over the extension algebra of generalised Steinberg representations:
\begin{equation}\label{eqn:main, intro}
\Mod(\Ext_{G}^\bullet(\Sigma_\lambda, \Sigma_\lambda))
\simeq 
\Perv_{\dualgroup{G}}^\circ(X_\lambda).
\end{equation}
This is the first main result of the paper, Theorem~\ref{thm:main1}. The equivalence suggests a Koszul duality between a full subcategory of $\textbf{Rep}(G)$ cut out by generalised Steinberg representations and $\Perv_{\dualgroup{G}}^\circ(X_\lambda)$. Indeed, we proceed to show that the extension algebra above is Koszul and that there is a Koszul duality of algebras
\[
\Ext_{G}^\bullet(\Sigma_\lambda, \Sigma_\lambda)^! \cong \Ext_{\dualgroup{G}}^\bullet(\mathcal{S}_\lambda, \mathcal{S}_\lambda), 
\]
with the latter algebra denoting the endomorphism algebra of $\mathcal{S}_\lambda$ in the $\dualgroup{G}$-equivariant derived category on $X_\lambda$, where  $\mathcal{S}_\lambda$ denotes the direct sum of representatives from each class of simple objects in the principal block of its perverse heart. 

\subsection{Motivation}

Studying the category of $\dualgroup{G}$-equivariant perverse sheaves on $X_\lambda$ itself, as opposed to the category of modules over its extension algebra, is important from the perspective of the Voganish project \cites{Mracek,CFMMX,CFK,CFZ:cubic,CFZ:unipotent,CR1,CR2,CDFZ:Generic, CDFZ:Whittaker}, whose aim is generalise Arthur packets using vanishing cycles of perverse sheaves, as conjectured in \cite{CFMMX}, based on \cite{Vogan:Langlands}, and demonstrated for $G(F) = \GL_n(F)$ in \cites{CR1,CR2}.
With this in mind, we believe that an understanding of the internal structure of the Abelian category $\Perv_{\dualgroup{G}}(X_\lambda)$ will lead to a better understanding of the geometric properties of equivariant intersection cohomology complexes corresponding to representations of Arthur type via the categorical interpretation of the local Langlands correspondence as it appears in this paper. 
Also, we expect the constructions in this paper can be generalised to a larger class of representations by combining the results here with a geometric perspective on endoscopy (building on \cite{ABV} and work in progress to prove Vogan's conjecture for classical groups generalising \cite{CR1} and \cite{CR2}) and techniques from \cite{CFMMX}.

\subsection{Relation to other work} 

In part, this work was inspired by insights gained from a conjecture of Wolfgang Soergel \cite{Soergel} that, for groups over $\mathbb{R}$, relates a full subcategory of $G(F)$ identified by a specific infinitesimal character to the extension algebra of simple equivariant perverse sheaves appearing on the appropriate Vogan variety. 
The analogue of Soergel's conjecture for $p$-adic groups suggests a relation between the extension algebra of the Vogan variety $X_\lambda$ attached to an infinitesimal parameter $\lambda$ and a full subcategory of $\Rep(G)$ attached to $\lambda$. 
In this paper, through Theorem~\ref{thm:main1} and Theorem~\ref{thm:main2} we demonstrate such a $p$-adic analogue to Soergel's conjecture, relating the extension algebra of representations with infinitesimal parameter $\lambda$ to the category of equivariant perverse sheaves on the Vogan variety for $\lambda$. In the final section, we then show that a version of Soergel's conjecture can be recovered by Koszul duality of the extension algebras on either side of the Langlands correspondence. 

\subsection{Notation and conventions} 
Throughout, let $G$ be a connected reductive algebraic group, semisimple and split over a local non-Archimedean field $F$. By $\Rep(G)$ we denote the category of smooth, complex representations of $G(F)$, and by $\textbf{Rep}_{\text{fl}}(G)$ the full subcategory of representations of finite length. The algebras appearing in this paper are associative algebras with identity over $\CC$. We write $\Mod(A)$ for the category of finite dimensional right modules over such an algebra $A$. By variety, we mean a reduced, separated scheme of finite type over $\CC$. When referring to closures of subvarieties, we mean the closure in the Zariski topology.  
\subsection{Acknowledgements}

We thank Wolfgang Soergel, Ahmed Moussaoui and Kristaps Balodis, for helpful conversations.

\section{Geometry of Steinberg representations}\label{section:spectral cat}
In this section, we study the representation-theoretic side of our correspondence, focusing on the extensions between generalised Steinberg representations of $ G(F) $. These may be understood as the ``maximally singular'' irreducible representations of $ G(F)$, since any irreducible representation in $\Rep(G)$ with non-trivial cohomology must be isomorphic to a generalised Steinberg representation \cite{Ca74}*{Theorem 2.c}. We explicitly describe the extension algebra generated by these representations, drawing on the work of \cite{Dat} and \cite{Orlik}, which provide a combinatorial computation of the extensions between generalised Steinberg representations, classified according to the subsets of simple roots of $G$ by which these representations are parametrised.

\subsection{Generalised Steinberg representations}\label{ssec:gneralisedSteinberg}
Fix a Borel subgroup $B \subset G$ and a maximal torus $T \subset B$ and let $S$ be the set of simple $F$-roots of $G$ determined by these choices. For any subset $I$ of $S$, let $P_I$ be the standard parabolic subgroup of $G$ so that the unipotent radical $U_I$ of $P_I$ contains the root spaces for all $\alpha \in I$; we write $M_I$ for the Levi of $P_I$. Then, $I\subseteq J$ if and only if $P_I \subseteq P_J$. For instance, $P_\emptyset = B$ while $P_{S} = G$. We denote by 
\[
\Ind_{P_I}^{G}: \textbf{Rep}(M_I) \rightarrow \textbf{Rep}(G)
\]
the usual (non-normalised) parabolic induction, being inflation from $M_I(F)$ to $P_I(F)$, followed by induction. Explicitly, any smooth representation $\pi = (\pi,V)$ of a Levi subgroup $M(F) \subseteq G(F)$ can be viewed as a representation of the corresponding parabolic $P(F)$ by letting the unipotent factor act trivially. Then, the $G(F)$-representation $\Ind_{P}^{G}(\pi)$ consists of the space of functions 
\[
f:G(F) \rightarrow V ,
\]
such that 
\begin{enumerate}
    \item $f$ is locally constant and compactly supported; 
    \item $f(pg) = \pi(p)f(g)$, for all $p \in P(F)$ and $g \in G(F)$. 
\end{enumerate}
\begin{definition}\label{definition:gensteinbergs}
For any $I \subseteq S$, the \emph{generalised Steinberg representation} attached to $I$ is 
\[
\sigma_I:= \Ind_{P_I}^{G}(\1_I) / \sum_{I \subsetneq J} \Ind_{P_J}^{G}(\1_J) ,
\]
where $\1_I$ is the trivial representation of $M_I(F)$ and the induction is non-normalised.
\end{definition}

Note that $\sigma_\emptyset$ is the usual Steinberg representation of $G(F)$, commonly denoted by $\St_G$,  and $\sigma_{S}$ is the trivial representation of $G(F)$, for which we use the notation $\1_G$. 
Note that this definition uses non-normalised parabolic induction. While these representations have been constructed in various ways, we follow \cite{Orlik}*{\S 1} and \cite{Dat}*{\S 1} in the description provided. 
\begin{lemma}
    Generalised Steinberg representations are irreducible, admissible, and $\sigma_I \cong \sigma_J$ if and only if $I = J$. Furthermore, all generalised Steinberg representations lie in a common Bernstein block of $\Rep(G)$ determined by our choice of torus for $G$ and its trivial representation.
\end{lemma}

\begin{proof}
Generalised Steinberg representations are admissible, as they are quotients of representations induced from admissible representations of a Levi subgroups.
By \cite{Ca}*{Theorem 1.1}, 
generalised Steinberg representations are irreducible, and further, $\sigma_I \cong \sigma_J$ if and only if $I = J$ (in which case the isomorphism is equality).  
We can see that all generalised Steinberg representations appear in the principal block since every such representation appears as a subquotient of the induced representation $\Ind_{P_\emptyset}^G(\1_\emptyset)$,
by definition, where $P_\emptyset = T$ is our choice of maximal torus and $\1_\emptyset = \1_T$ is its trivial representation. A trivial representation of the torus is supercuspidal, hence this implies that all generalised Steinberg representations appear in the Bernstein block associated with the pair $[T,\1_T]$.
\end{proof}

%\begin{definition}
The following notation will be used in the next result and in the remainder of this paper:
for $I,J \subseteq S$, the \emph{symmetric difference} of $I$ and $J$ is the subset 
\[
I \ominus J := (I \cup J) \setminus (I \cap J),
\]
whose cardinality we denote as $\delta(I,J) := | I \ominus J | $. 
%\end{definition} 

\begin{proposition}\cite{Dat}*{Theorem~1.3} \label{prop:dat} 
Let $I ,J, K \subseteq S$ be subsets of simple roots for a split semi-simple group $G$. 
\begin{enumerate}
\item The corresponding generalised Steinberg representations $\sigma_I$ and $\sigma_J$ admit extensions given by the formula
\[
\dim \Ext_G^{k}(\sigma_I, \sigma_J) = 
\begin{cases} 
1 &\text{ if } k = \delta(I,J) \\
0 & \text{ otherwise.} 
\end{cases} 
\]
\item The cup product
\[
\Ext_G^{i }(\sigma_I, \sigma_J) \otimes \Ext_G^{j }(\sigma_J, \sigma_K) \rightarrow \Ext_G^{i + j}(\sigma_I, \sigma_K) 
\]
is nonzero if and only if $i = \delta(I,J)$, $\delta(J,K) = j$, and $i + j = \delta(I,K)$, in which case it is an isomorphism. 
\end{enumerate} 
\end{proposition}

\begin{proof}
Result (1) and the forward direction of (2) are given explicitly by \cite{Dat}*{Theorem~1.3} (the former was also proved independently in \cite{Orlik}*{Theorem 1}). Meanwhile, the converse statement in (2) is a straightforward application of the definition of the cup product to (1). 
\end{proof}

\subsection{Aubert-Zelevinsky duality}\label{ssection:AZ}
Here we recall basic properties of a pair of functors that shed light on $\Rep(G)$: normalised parabolic induction and Aubert-Zelevinsky duality (the latter sometimes referred to as Aubert-Zelevinky involution or cohomological duality). Throughout, we still assume that $G$ is split semisimple, though, of course, the results hold under much weaker constraints. 

%\begin{definition}
Following \cite{NP}*{\S 2}, we refer to
    \[
\AZ_G:= \Ext_G^d(-, \mathcal{H}(G)) \circ (-)^\vee : \textbf{Rep}_{\text{fl}}(G) \rightarrow \textbf{Rep}_{\text{fl}}(G),
    \]
as the \emph{Aubert-Zelevinsky duality functor} for $G$, where $d = |S|$ and
 \[
 \Ext_G^d(-, \mathcal{H}(G)):\textbf{Rep}_{\text{fl}}(G) \rightarrow \textbf{Rep}_{\text{fl}}(G)^{\text{op}}
 \]
is Bernstein's cohomological duality functor built from the Hecke algebra $\mathcal{H}(G)$ (cf., \cite{BR}*{\S 5.1}), and where 
\[
(-)^\vee: \textbf{Rep}_{\text{fl}}(G)  \rightarrow \textbf{Rep}_{\text{fl}}(G)^{\text{op}}
\]
is the usual contragradience duality for admissible representations of $G$. We say that $\text{A\hskip-1ptZ}(\pi)$ is the \textit{Aubert-Zelevinksy dual} of $\pi$, for any $\pi \in \textbf{Rep}_{\text{fl}}(G)$. In cases where there can be no confusion, we write $\text{A\hskip-1ptZ} = \text{A\hskip-1ptZ}_G$\footnote{In fact, the description of this involution as a functor is due to Bernstein \cite{BR}*{\S 5}, though it coincides with \cite{Aubert}*{Definition 1.5} when passing to the Grothendieck group. We use the term ``Aubert-Zelevinsky" duality  in accordance with \cite{NP}*{\S 2}, which provides a useful overview of the various incarnations of this phenomenon.}.
%\end{definition}

As we will need to make use of both normalised and non-normalised parabolic induction, we use the notation
\[
i_P^G(-) := \Ind_P^G( \delta_P^{1/2} \otimes (-)) : \textbf{Rep}(M) \rightarrow \textbf{Rep}(G)
\]
for normalised parabolic induction from $M(F)$ to $G(F)$, where $P = MU$ is the Levi decomposition of the appropriate parabolic. 
\begin{proposition}\label{prop:aubertproperties}
Aubert-Zelevinsky duality enjoys the following properties: 
\begin{enumerate}
\item $\AZ$ is exact; 
    \item $\AZ \circ \AZ = \text{id}_{\textbf{Rep}(G)}$, \emph{i.e}., $\AZ$ is an idempotent autoequivalence; 
    \item $\AZ$ fixes supercuspidal representations; 
    \item For any Levi subgroup $M \subseteq G$, with parabolic $P = MU$, \[\AZ_G \circ i_P^G =  i_{\overline{P}}^G \circ \AZ_M,
    \]
    where ${\bar P}$ is the opposite parabolic to $P$.
    \item $\AZ(\chi \otimes \pi) \cong \chi \otimes \AZ(\pi)$, where $\pi$ is an irreducible representation and $\chi$ is any character. 
\end{enumerate}
\end{proposition}
\begin{proof}
    Statements (1)-(4) can be calculated by applying well known properties of contragradience to \cite{BR}*{Theorem 31}. For (5), from \cite{NP}*{Theorem 2} it follows that $\AZ(\pi)$ can be characterised as the irreducible representation, unique up to isomorphism, so that 
    \[
\dim \Ext_G^{d(\pi)}(\pi, \AZ(\pi))  = 1,
    \]
    where $d(\pi)$ is the split rank of cuspidal support for $\pi$, given by $d = |S|$ in our case. Now, the functor
    \[
\chi \otimes - : \Rep(G) \rightarrow \Rep(G) 
    \]
    is an equivalence of categories, since its inverse is given by tensoring with $\chi^{-1}$. Any equivalence of Abelian categories induces an equivalence of the corresponding derived categories, and thus we have
    \[
\Ext_G^{d(\pi)}(\pi, \AZ(\pi)) \cong \Ext_G^{d(\pi)}(\chi \otimes \pi, \chi \otimes \AZ(\pi)). 
    \]
    Since $d(\pi) = d(\chi \otimes \pi)$, this same characterization of Aubert-Zelevinsky duality implies that
    \[
\chi \otimes \AZ(\pi) \cong \AZ(\chi \otimes \pi) 
    \]
    which is what we wanted to show. 
\end{proof}

\begin{lemma}\label{lem:aubertdualofsteinberg} 
    For any $I \subseteq S$, the Aubert-Zelevinsky dual of the corresponding generalised Steinberg representation is given by 
    \[
\AZ(\sigma_I) \cong \sigma_{I^c} 
    \]
    for $I^c \ceq S\setminus I$. 
\end{lemma}

\begin{proof}
By \cite{NP}*{Theorem 2}, the Aubert-Zelevinsky dual of $\AZ(\sigma_I)$ is isomorphic to any irreducible representation $\pi$ satisfying
\[
\Ext_G^d(\sigma_I, \pi) \neq 0 
\]
where $d = |S|$. From Proposition~\ref{prop:dat}, if $\Ext_G^k(\sigma_I, \sigma_{I^c} ) \neq 0 $, then $k = \delta(I,J)$, which is to say that
\[
k = | I \ominus I^c | = | (I \cup I^c ) \setminus (I \cap I^c)| = | S | = d
\]
as required. 
\end{proof} 

Note that $d = |S|$ is in fact the unique maximal natural number $d$ so that 
\[
\Ext_G^d( \pi , \pi') \neq 0 
\]
for any other irreducible representation $\pi'$, and $\pi' = \AZ(\pi)$ in this case \cite{NP}*{Theorem 2}.

\subsection{Standard representations}

\begin{proposition}\label{proposition:standard}
Let $I$ be a subset of $S$ and let $I^c$ be its complement in $S$. Let $\delta_{\overline{P}_{I^c}}^{1/2}$ be the modulus character for the parabolic subgroup ${\overline{P}_{I^c}}$, which denotes the opposite parabolic to $P_{I^c}$.
The standard representation for $\sigma_I$ is 
    \[
 \Delta(\sigma_I) 
 =
 i_{{\bar P}_{I^c}}^G
 \left(
 \delta^{1/2}_{{\bar P}_{I^c}}\otimes \St_{M_{I^c}}
 \right)
    \]
    where $St_{M_{I^c}}$ is the Steinberg representation for the Levi subgroup $M_{I^c} \subseteq G$; in particular, $\sigma_I$ is the unique irreducible quotient of $\Delta(\sigma_I)$. 
\end{proposition}

\begin{proof}
By definition of $\sigma_{I^c}$, there is a quotient
\[
\Ind_{P_{I^c}}^G(\mathbbm{1}_{M_{I^c}}) \twoheadrightarrow \sigma_{I^c}. 
\]
Since $\AZ$ is exact, applying it this map gives the new quotient 
\[
\AZ(\Ind_{P_{I^c}}^G(\mathbbm{1}_{M_{I^c}})) \twoheadrightarrow \sigma_I 
\]
since $\AZ(\sigma_{I^c}) = \sigma_I$, by Lemma~\ref{lem:aubertdualofsteinberg}. Taking the codomain of the map, we use the properties of $\AZ$ stated in Proposition~\ref{prop:aubertproperties} to calculate
\begin{align*}
    \AZ_G(\Ind_{P_{I^c}}^G(\mathbbm{1}_{M_{I^c}})) & = \AZ_G(i_{P_{I^c}}^G(\delta_{P_{I^c}}^{-1/2} \otimes \mathbbm{1}_{M_{I^c}}) )\\
    & = i_{{\bar P}_{I^c}}^G(\AZ_M(\delta_{P_{I^c}}^{-1/2}) \otimes \AZ_M(\mathbbm{1}_{M_{I^c}}) ) \\
    & = i_{{\bar P}_{I^c}}^G(\delta_{P_{I^c}}^{-1/2} \otimes \St_{M_{I^c}}) ),
\end{align*}
where the last equality follows from the well know fact that the Aubert-Zelevinsky dual of trivial representation is the Steinberg representation, and that $\delta_{P_{I^c}}^{-1/2}$ is a character, and hence supercuspidal and fixed by Aubert. But $\delta_{P_{I^c}}^{-1/2} = \delta_{{\bar P}_{I^c}}^{1/2}$, which is ${\bar P}_{I^c}$-positive. The quotient then becomes 
\[
 i_{{\bar P}_{I^c}}^G(\delta_{{\bar P}_{I^c}}^{1/2} \otimes \St_{M_{I^c}}) ) \twoheadrightarrow \sigma_I. 
\]
Since $\delta_{{\bar P}_{I^c}}^{1/2}$ is ${\bar P}_{I^c}$-positive and $\St_{M_{I^c}}$ is tempered, this is a standard representation \cite{Konno:Langlands}*{\S 3.1} and must be the standard representation for $\sigma_I$, by the existence of the above quotient. 
\end{proof}

\subsection{Langlands parameters}\label{subsection:parameters}

To make language precise, by \emph{Langlands parameter} we mean a group homomorphism 
\[
\phi: W_F' \rightarrow \Lgroup{G},
\]
with $W_F' := W_F \times \SL_2(\C)$, satisfying conditions appearing in \cite{Borel:Corvallis}*{\S 8.2} (see also \cite{CFMMX}*{\S 3.4})
and we reserve the term \emph{L-parameter} for the $\dualgroup{G}$-conjugacy class of a Langlands parameter. 
By an \emph{infinitesimal parameter} we mean a group homomorphism
\[
\lambda: W_F \rightarrow {}^L G 
\]
which can be viewed as a Langlands parameter trivial on $\SL_2(\CC)$.
We say that $\lambda$ is the infinitesimal parameter for a Langlands parameter $\phi$ if the following diagram commutes
\[\begin{tikzcd}[ampersand replacement=\&]
	{W_F':= W_F\times \text{SL}_2(\mathbb{C})} \&\& {{}^LG} \\
	{W_F}
	\arrow["\phi", from=1-1, to=1-3]
	\arrow[from=2-1, to=1-1]
	\arrow["\lambda"', from=2-1, to=1-3]
\end{tikzcd}\]
where the vertical map is given by $w \mapsto \left( w, d_w \right)$, where 
\begin{equation}\label{def:dw}
d_w := \diag\left(|w|_F^{1/2}, |w|_F^{-1/2} \right).
\end{equation}

\begin{remark}
    In the same way that one might collect irreducible representations of $G(F)$ that appear in a common induced representation from a Levi subgroup by recognising that such representations share a common character of the Bernstein centre (\emph{i.e.}, a ring homomorphism from the endomorphism ring of the identity functor on $\textbf{Rep}(G)$ to the complex numbers), we can collect Langlands parameters by studying those such parameters that share a common infinitesimal parameter upon pullback to $W_F$. In fact, these two organisational methods are compatible with the Langlands correspondence via the map 
$
\mathfrak{X}_G \rightarrow \mathfrak{Y}_G^{\text{st}}, 
$
where $\mathfrak{X}_G := \Spec{\End(\1_{\textbf{Rep}(G)})}$ and where
where $\mathfrak{Y}_G^{\text{st}}$ is the so-called \textit{stable Bernstein centre}, the variety consisting of $\dualgroup{G}$-conjugacy classes of infinitesimal parameters $W_F \rightarrow {}^LG$. Then, two Langlands parameters share a common infinitesimal parameter if and only if the irreducible representations appearing in their $L$-packets share a common character of the Bernstein centre \cite{Haines}*{Proposition 4.23}.
\end{remark}

\begin{lemma}\label{lemma:Lparameter}
The L-parameter $\phi_{I}$ for $\St_I$ is  
\[
\phi_{I} = \iota_{\dualgroup{M}_{I^c}} \circ \left(\delta^{1/2}_{{\bar P}_{I^c}}\otimes \phi_{\St_{M_{I^c}}} \right),
\]
where $\delta_{{\bar P}_{I^c}}$ is the modulus character of $\dualgroup{M}_{I^c}$ for the parabolic subgroup ${\bar P}_{I^c}$.
Here we write $\iota_{\dualgroup{M_{I^c}}}$ for the inclusion of $\dualgroup{M_{I^c}}$ into $\dualgroup{G}$.
Moreover, all generalised Steinberg representations share the same infinitesimal parameter
\[
 \lambda(w) = 
\diag\left(\abs{w}^{d/2}, \abs{w}^{(d-2)/2}, \ldots, \abs{w}^{-(d-2)/2},\abs{w}^{-d/2}\right).
\]
\end{lemma}

\begin{proof}
The generalised Steinberg representation $\sigma_I$ is of Arthur type only for $I = \emptyset$ and $I = S$.
The Steinberg representation $\sigma_\emptyset = \St_G$ of $G(F)$ has Arthur parameter 
\[
\psi_{\St_G}(w,x,y) = \Sym^{d}(x),
\]
while the trivial representation $\1_G = \sigma_{S}$ of $G(F)$ has Arthur parameter 
\[
\psi_{\1_G}(w,x,y) = \Sym^{d}(y).
\]
(Recall that $\Sym^{d}$ is $d+1$-dimensional representation.)
The infinitesimal parameter of $\St_{M}$ is therefore
\[
\lambda(w) = \Sym^{d}(d_w) \in \dualgroup{T},
\]
where $d_w$ is given by Equation~\ref{def:dw}.
To find the Langlands parameter of $\pi_I$, we recall the standard representation for $\pi_I$ from Proposition~\ref{proposition:standard}, and then use \cite{Silberger}*{\S 7.2 and Introduction, p. 301}.
\end{proof}

\begin{lemma}\label{lemma:enhancement}
    The enhanced Langlands parameter (in the sense of Vogan's version of the local Langlands correspondence) for $\sigma_I$ is $(\phi_I, \1)$ where $\1$ denotes the trivial representation of the component group $\pi_0\left(Z_{\dualgroup{G}}(\phi_I)\right)$.
\end{lemma}
\begin{proof}
    This is a consequence of the Whittaker normalization of L-packets, together with the fact that the Steinberg representation is generic.
\end{proof}
We have now established a representation-theoretic foundation for our correspondence, showing that generalised Steinberg representations correspond to a well-behaved family of Langlands parameters collected by a common infinitesimal parameter. Next, we view this foundation through the lens of Vogan's geometric reinterpretation of the local Langlands correspondence. The construction of this perspective mirrors many of the algebraic structures present on the automorphic side of the correspondence, the combinatorics of the extension algebra for generalised Steinberg representations in particular, which will eventually lead to the desired Koszul duality. 

\subsection{"Vogan variety" of Langlands parameters}\label{ssec:Vogan variety}
Rather than considering all Langlands parameters attached to $G(F)$, David Vogan's geometric reimagining of the Langlands correspondences instead considers only those Langlands parameters that share a common infinitesimal parameter $\lambda:W_F \rightarrow {}^L G$ upon pullback to the Weil group by $w \mapsto d_w$ given by Equation~\ref{def:dw}. 
The set $X_\lambda$ of all such parameters  naturally enjoys the structure of an algebraic variety equipped with a $\dualgroup{G}$-action, the same conjugacy action that collects Langlands parameters into the equivalence classes deemed $L$-parameters. 
As such, the Langlands parameters living on a common $\dualgroup{G}$-orbit of $X_\lambda$ share a common $L$-packet, inducing a bijection between the $L$-packets attached to Langlands parameters in $X_\lambda$ and the $\dualgroup{G}$-orbits of this variety. 
As only finitely many $L$-parameters share any given infinitesimal parameter $\lambda$, the variety $X_\lambda$ has only finitely $\dualgroup{G}$-orbits under this action. 
In fact, the simple, equivariant perverse sheaves of the orbit stratification recovers the representation theory of $G(F)$ for irreducible representations living inside of these $L$-packets. 

At first glance, one would hope to find a canonical bijection between the smooth, irreducible representations living in an $L$-packet attached to a Langlands parameter $\phi$ and the irreducible, equivariant local systems appearing on its orbit in $X_\lambda$. However, such a formulation proves too optimistic, as there are generally more isomorphism classes of irreducible, equivariant local systems on the orbit than there are isomorphism classes of irreducible representations in the corresponding packet. This imprecision is resolved by the realisation that these additional local systems correspond not to irreducible representations of $G(F)$ itself, but to irreducible representations of pure inner forms of $G(F)$ that share this same infinitesimal parameter. Using the Kottwitz isomorphism \cite{Kottwitz-iso}*{Proposition 6.4}, these pure inner forms are indexed by isomorphism classes of characters $\pi_0(Z(\dualgroup{G})^{\Gamma_F})\rightarrow \CC$ (in our case $\Gamma_F$ acts trivially). 
Therefore, if we replace the $L$-packet associated to $\phi$ with its \textit{pure $L$-packet}, being the union of all $L$-packets of $\phi$ for each pure inner form of $G(F)$, denoted 
\[
\Pi_\phi^{\text{pure}}(G) := \bigsqcup_{\delta \in H^1(F,G)} \Pi_\phi\left(G_\delta\right), 
\]
where $G_\delta$ is the pure inner form for $\delta$, Vogan's work can be used to give a canonical bijection 
\[
\bigsqcup_{\phi|_{W_F} = \lambda} \Pi_\phi^{\text{pure}}(G) \xrightarrow{\sim} \text{Irr}\left(\Perv_{\dualgroup{G}}(X_\lambda)\right).
\]
For our purposes, however, we wish to understand how the extensions between irreducible admissible representations of $G(F)$ are realised within the geometry of $X_\lambda$, by which we mean the within the category of $\dualgroup{G}$-equivariant perverse sheaves on this variety. To accomplish this, we identify a full Abelian subcategory
\[
\Perv_{\dualgroup{G}}^\circ(X_\lambda) \hookrightarrow \Perv_{\dualgroup{G}}(X_\lambda).
\]
whose isomorphism classes of simple objects are exactly those simple objects in $ \Perv_{\dualgroup{G}}(X_\lambda)$ that correspond to representations of $G(F)$ itself rather than its pure inner forms. That is, we wish to identify such a subcategory so that Vogan's bijection restricts to a bijection 
\[
\bigsqcup_{\phi|_{W_F} = \lambda} \Pi_\phi(G) \xrightarrow{\sim} \text{Irr}\left(\Perv_{\dualgroup{G}}^\circ(X_\lambda)\right)
\]
In due course, we will go through a careful description of this subcategory in the case of Langlands parameters attached to generalised Steinberg representations.
The simple $\dualgroup{G}$-equivariant perverse sheaves on $X_\lambda$ outside of this subcategory also carry spectral data, and correspond to certain irreducible representations of particular pure inner forms of $\dualgroup{G}$. For a complete treatment of this construction, see \cite{Vogan:Langlands}*{\S 4} and \cite{CFMMX}*{Chapter 4}.
\\

While Vogan's construction successfully geometrises the ``Galois side" of the Langlands correspondence, the variety $X_\lambda$ is often rather hard to work with in practice; it is generally disconnected, in particular. To solve this problem, we build a subvariety $V_\lambda \subset \hat{\mathfrak{g}}$, the Lie algebra for $\dualgroup{G}$, that retains all the relevant geometric data (and, therefore, spectral data, under Vogan's philosophy). 
For any infinitesimal parameter $\lambda : W_F \rightarrow \Lgroup{G}$, this subvariety is referred to as the associated \textit{Vogan variety} and is the prehomogeneous vector space
\[
V_\lambda := \{x \in \hat{\mathfrak{g}} \mid \Ad(\lambda(w)) x = |w|_F\ x,\ \forall w\in W_F\},
\]
which comes equipped with the natural conjugation action of the group 
\[
H_\lambda := \{g \in \dualgroup{G} \mid \lambda(w) g = g \lambda(w),\ \forall w \in W_F \}.
\]
The group $H_\lambda$ is reductive, but not necessarily connected.
The $H_\lambda$-orbits in $V_\lambda$ likewise correspond bijectively to the complete list of L-parameters sharing a common infinitesimal parameter $\lambda$. 
Given this description, we can recover our initial parameter variety
\[
X_\lambda = (\dualgroup{G}\times V_\lambda)/H_\lambda,
\]
the quotient space produced by the action $h(g,x) = (gh^{-1}, \Ad(h)x)$.
As promised, there is an equivalence of derived categories given by the usual induction equivalence
\[
D^b_{H_\lambda} (V_\lambda) \simeq D^b_{\dualgroup{G}}(X_\lambda), 
\]
which restricts to the equivalence 
\[
\Perv_{H_\lambda} (V_\lambda) \simeq \Perv_{\dualgroup{G}}(X_\lambda),  \quad \Perv_{H_\lambda}^\circ (V_\lambda) \simeq \Perv_{\dualgroup{G}}^\circ(X_\lambda)
\]
so that we may pick out the simple $H_\lambda$-equivariant perverse sheaves on $V_\lambda$ mapping to the appropriate irreducible representations of $G$ under Vogan's conception of the local Langlands correspondence. In practice, $V_\lambda$ is often easier to work than with than $X_\lambda$ since it is not only a subvariety of the Lie algebra, but a prehomogeneous vector space \cite{Sato-Kimura}*{Definition 2.1}. Both properties of $V_\lambda$ provide numerous tools for calculation not enjoyed by $X_\lambda$, and, as such, it is often useful to perform these calculations on $V_\lambda$ first and proceed to infer the analogous properties on $X_\lambda$ via the equivalence of equivariant derived categories. \\

Fix a Borel $\dualgroup{B} \subset \dualgroup{G}$ and a maximal torus $\dualgroup{T} \subset \dualgroup{B}$, and let $\dualgroup{S}= \{ {\hat \alpha}_1, \ldots, {\hat \alpha}_d\} $ denote the set of simple roots for this choice with a fixed ordering. 
For generalised Steinberg representations, the Vogan variety $V_\lambda$ is the sum of the root spaces attached to the simple roots in $\dualgroup{G}$, for the choice of Borel above: 
\begin{equation}\label{eqn:VtoAd}
V_\lambda = \mathop{\oplus}\limits_{\hat{\alpha} \in \dualgroup{S}}   \mathfrak{\hat g}_{{\hat \alpha}} \iso \mathbb{A}^{d},
\end{equation}
equipped with the action of 
\[
H_\lambda = \dualgroup{T} \iso \mathbb{G}_{\hskip-1pt \text{m}}^d,
\]
given by, for $t\in \hat{T}$ and $(x_1, \dots, x_d) \in \Gmult^d$, 
\[
t\cdot (x_1, \ldots , x_d) = ({\hat \alpha}_1(t) x_1, \ldots , {\hat \alpha}_d(t)x_d). 
\]
Here we have used the fact that $\dualgroup{G}$ is semisimple, since passing to the Langlands dual group preserves semisimplicity.

It is not difficult to see that $H_\lambda$-orbits in $V_\lambda$ are also parametrised by subsets of $S$.
In fact, the variety $V_\lambda$ with this action by $H_\lambda = \dualgroup{T}$ appears in \cite{Vogan:Langlands}*{Example 4.9}, where it is remarked that every $H_\lambda$-orbit $C \subset V_\lambda$ is isomorphic to a variety of the form
    \[
C \cong C_1 \times C_2 \times \dots \times C_d 
    \]
where either $C_i = \{0\}$ or $C_i =\Gmult$, for every $1 \leq i \leq d$.
In particular, there are $2^d$ orbits, with orbit representatives $C_I$, for $I\subseteq \dualgroup{S}$. From now on, we set $C_I^i = \{0\}$ if $\alpha_i \in I$, and $C_I^i = \Gmult$ otherwise.

\begin{lemma}
    If $\phi$ has infinitesimal parameter $\lambda$, then $\phi \iso \phi_I $, for some $I\subseteq S$. 
\end{lemma}

\begin{proof}
Order $\dualgroup{S}$ as above.
Fix non-zero $u_{{\hat \alpha}_i}$ in the root space $ \mathfrak{\hat g}_{{\hat \alpha}_i} \subset \mathfrak{\hat g}$ for ${\hat \alpha}_i$. For every $I\subseteq S$, let $x_I = (x_1, \ldots, x_d)$ be the binary digit for $I^c$, given by 
\[
x_i = \begin{cases} 
0 &  \alpha_i\in I\\
1 & \alpha_i\notin I.
\end{cases}
\]
Set $x_I \ceq \sum_{i=1}^d x_i u_{{\hat \alpha}_i}\in \mathfrak{\hat g}$.
Then the unramified Langlands parameter $\phi_I$ is uniquely determined by the two properties
\[
\phi_I(\Frob,d_{\Frob}) = \lambda(\Frob)
\qquad\text{and}\qquad
\phi_I(1,e) = (\exp(x_I), 1),
\]
where $d_{\text{Fr}} = 
\text{diag}(q^{1/2}, q^{-1/2})$, and 
$e = \left(\begin{smallmatrix} 1 & 1 \\ 0 & 1 \end{smallmatrix}\right)$. 
As in \cite{CFMMX}*{Proposition 4.2}, this establishes a map from Langlands parameters with infinitesimal parameter $\lambda$ to points in $V_\lambda$, which then induces a bijection between $L$-parameters and $H_\lambda$-orbits in $V_\lambda$. 
\end{proof}

Because this context is more amenable to calculation, we will build our framework using equivariant perverse sheaves on the $H_\lambda$-variety $V_\lambda$, eventually translating it back to the setting of the parameter $\dualgroup{G}$-variety $X_\lambda$ via the equivalence of categories stated above.

\section{The principal block of equivariant perverse sheaves}\label{section:PS}

In Section~\ref{ssec:gneralisedSteinberg} we outlined an algebraic description of the extensions of generalised Steinberg representations with a view towards understanding the extension algebra they produce. Now, we turn our attention towards their geometric counterpart in Vogan's framework by constructing the full subcategory $\Perv_{H_\lambda}^\circ(V_\lambda)$ of $\Perv_{H_\lambda}(V_\lambda)$, referred to as the \emph{principal block} of $\Perv_{H_\lambda}(V_\lambda)$.
As discussed in the previous section, this subcategory should be thought of as the full subcategory of equivariant perverse sheaves whose simple intersection cohomology complexes correspond to irreducible representations of the group $G(F)$ itself, rather than its pure inner forms. This subcategory is produced by demonstrating that $H_\lambda$ acts on $V_\lambda$ through an isogeny to another algebraic that, taken in isolation, acts on $V_\lambda$ via an action whose kernel is ``smaller" (in fact trivial in our case) than that of $H_\lambda$. While we give the description solely for the case of generalised Steinberg representations, a similar process should collect the smooth, irreducible representations of the quasi-split form attached to a much more general class of infinitesimal parameters for a much more general class of algebraic groups. 

\subsection{Defining the principal block}\label{subsection:principal block}

\begin{proposition}\label{prop:embedding}
The action of $H_\lambda$ on $V_\lambda \cong \mathbb{A}^d$ factors through an isogeny 
\[
H_\lambda \twoheadrightarrow \Gmult^d
\]
where $\Gmult^d$ acts on $\mathbb{A}^d$ component-wise. This provides an embedding of categories
\[
\Perv_{\Gmult^d}(\mathbb{A}^d) \hookrightarrow \Perv_{H_\lambda}(V_\lambda) .
\]
\end{proposition}

\begin{proof}
Define the group homomorphism $H_\lambda \cong \dualgroup{T} \to \Gmult^d$ given by
\begin{equation}\label{equation:surjection} 
t \mapsto ({\hat \alpha}_1(t), \ldots, {\hat \alpha}_d(t)) 
\end{equation}
This map is equivariant for the actions of the two groups on $V_\lambda$, and surjective, which can be seen by examining the rank of the derivative of $H_\lambda \cong \dualgroup{T} \to \Gmult^d$.
Since $V_\lambda \cong \mathbb{A}^d$ by Equation~\eqref{eqn:VtoAd}, the result follows from Lemma~\ref{lemma:embedding}, below. 
\end{proof}

\begin{definition}\label{definition:principal block}
    We denote the essential image of the functor in Proposition~\ref{prop:embedding} by $\Perv_{H_\lambda}^\circ(V_\lambda)$ and refer to it as the {\it principal block} of the category $\Perv_{H_\lambda}(V_\lambda)$.
\end{definition}
Note that we use different notation for the group $H_\lambda$ and $\Gmult^d$ (resp. varieties $V_\lambda$ and $\mathbb{A}^d$, despite the fact that they are isomorphic to emphasise difference between the actions. However, as we shall soon see, the action of $\Gmult^d$ on $\mathbb{A}^d$ nonetheless produces the same stratification as the usual Vogan variety, with the essential difference being the equivariant fundamental groups which are trivial in this case (the subject of Proposition~\ref{proposition:simpleobjects}, below). 
\begin{remark}
    The isogeny $H_\lambda  \cong \dualgroup{T} \to \Gmult^d$ appearing in the proof of Proposition~\ref{prop:embedding} has kernel of order $2$ if $G$ is of type B or D.
    As the morphism $H_\lambda \to \Gmult^d$ generally does not have connected fibres, the principal block $\Perv_{H_\lambda}^\circ(V_\lambda)$ is a proper subcategory of $\Perv_{H_\lambda}(V_\lambda)$ corresponding to the generalised Steinberg representations of $G$ itself, rather than its pure inner forms. 
\end{remark}

\begin{lemma}\label{lemma:embedding} 
Let $m_0: G_0 \times X \rightarrow X$ and $m_1: G_1 \times X \rightarrow X$ be a pair of actions of the connected algebraic groups $G_0$ and $G_1$ on a variety $X$. Assume there is a surjective morphism of algebraic groups $h: G_1 \rightarrow G_0$ so that the diagram
\[
\begin{tikzcd}[ampersand replacement=\&]
	{G_1\times X} \&\& {G_0\times X} \\
	\& X
	\arrow["{h \times \text{id}_X}", from=1-1, to=1-3]
	\arrow["{m_1}"', from=1-1, to=2-2]
	\arrow["{m_0}", from=1-3, to=2-2]
\end{tikzcd}
\]
commutes. Then, there exists a fully-faithful embedding of categories 
\[
\Perv_{G_0}(X) \hookrightarrow \Perv_{G_1}(X)
\] 
If $h$ is surjective and smooth with connected kernel, then the embedding is an equivalence. 
\end{lemma}

\begin{proof}
Applying the product functor $(-) \times X$ to $h$, there is a surjective morphism of varieties $(h \times \text{id}_X): G_1 \times X \rightarrow G_0 \times X$. Clearly, we get the factorizations $m_1 = m_0 \circ  (h \times \text{id}_X)$ and $p_1 = p_0 \circ (h \times \text{id}_X) $, where $p_0$ and $p_1$ are the obvious projection maps. Now, recall that a $G_0$-equivariant perverse sheaf on $X$ is a perverse sheaf $\mathcal{F}$ equipped with an isomorphism \[\varepsilon: m_0^* \mathcal{F} \xrightarrow{\sim} p_0^* \mathcal{F}.\] Taking the pullback along $(h \times \text{id}_X)$, we get
\begin{align*}
&(h \times \text{id}_X)^*\varepsilon: m_0^* \mathcal{F} \rightarrow p_0^* \mathcal{F} & \\
& = (h \times \text{id}_X)^*\varepsilon: (h \times \text{id}_X)^* m_0^* \mathcal{F} \rightarrow (h \times \text{id}_X)^* p_0^* \mathcal{F}\\
& = (h \times \text{id}_X)^*\varepsilon: ( m_0 \circ (h \times \text{id}_X))^* \mathcal{F} \rightarrow (p_0 \circ (h \times \text{id}_X))^* \mathcal{F} \\
& = (h \times \text{id}_X)^*\varepsilon: (m_1)^* \mathcal{F} \rightarrow (p_1)^* \mathcal{F} 
\end{align*} 
Since any functor preserves isomorphisms, the pullback \[(h \times \text{id}_X)^*\varepsilon: (m_1)^* \mathcal{F} \xrightarrow{\sim} (p_1)^* \mathcal{F}\] is also an isomorphism, showing that $(\mathcal{F}, (h \times \text{id}_X)^*\varepsilon)$ is a $G_1$-equivariant perverse sheaf on $X$ whenever $(\mathcal{F}, \varepsilon)$ is a $G_0$-equivariant perverse sheaf on the same space. This defines a functor
\[
\Perv_{G_0}(X) \rightarrow \Perv_{G_1}(X); \quad (\mathcal{F}, \varepsilon) \mapsto (\mathcal{F}, (h \times \text{id}_X)^*\varepsilon)
\]
which we can define to be trivial on maps, since both $G_0$ and $G_1$ are assumed to be connected, and hence both $\Perv_{G_0}(X)$ and $\Perv_{G_1}(X)$ are full subcategories of the category of (not-necessarily-equivariant) perverse sheaves $\Perv(X)$ \cite{Achar}*{Proposition 6.2.15}. This means that homsets in both subcategories are the same as those in $\Perv(X)$, implying that our defined functor is fully-faithful, \emph{i.e.}., an embedding. 

Now, assume that $h$ is surjective and smooth with connected kernel. Let $\mathcal{F} \in \Perv_{G_1}(X)$. 
Then, by definition there is an isomorphism 
\[
\epsilon: m_1^* \mathcal{F} \xrightarrow{\sim}  p_1^* \mathcal{F} 
\]
Since $m_1 = m_1 \circ (h \times \text{id}_X)$, we can rewrite this as 
\[
\epsilon: (h\times \text{id}_X)^* m_0^* \mathcal{F} \xrightarrow{\sim} (h \times \text{id}_X)^* p_0^* \mathcal{F} 
\]
The map $(h \times \text{id}_X)$ is also smooth and surjective with connected kernel. 
Hence, by \cite{Achar}*{Lemma~3.7.3}, there are isomorphisms of functors
\[
(h \times \text{id}_X)_*(h \times \text{id}_X)^* m_0^* \cong m_0^* ; \quad \text{ and } (h \times \text{id}_X)_*(h \times \text{id}_X)^* p_0^* \cong p_0^* 
\]
Applying the functor $(h \times \text{id}_X)_*$ to $\epsilon$, we get an isomorphism 
\[
(h \times \text{id}_X)_* \epsilon: m_0^* \mathcal{F} \xrightarrow{\sim}  p_0^* \mathcal{F} , 
\]
showing that $\mathcal{F} \in \Perv_{G_0}(X)$ as well, and so both categories share the same objects. As we saw before, both categories also share homsets and the groups are assumed to be connected, thus implying the stated equivalence. 
\end{proof} 

\subsection{Simple objects in the principal block}\label{ssec:SimpleObjectsPrincipalBlock}

\begin{proposition}\label{proposition:simpleobjects} 
Every simple object in the category $\Perv_{H_\lambda}^\circ(V_\lambda)$ is isomorphic to a perverse sheaf of the form $\IC(\mathbbm{1}_C)$, where $C \subset V_\lambda$ is an $H_\lambda$-orbit. 
\end{proposition}

\begin{proof} 
By the equivariant version of \cite{BBD}*{Théorème 4.3.1}, the simple objects of $\Perv_{H_\lambda}^\circ(V_\lambda) \simeq \Perv_{\Gmult^d}(\mathbb{A}^d)$ are isomorphic to the perverse sheaves of the form $\IC(\mathcal{L})$ where $\mathcal{L}$ is an irreducible, $\Gmult^d$-equivariant local system on an orbit $C \subset V_\lambda$
 \cite{CFMMX}*{\S~4.6}.
 By \cite{Achar}*{Proposition~6.2.13}, these local systems are in bijection with the irreducible representations of the group 
\[
\pi_1^{\Gmult^d}(C,x) := \text{Stab}_{\Gmult^d}(x)/\text{Stab}_{\Gmult^d}(x)^\circ  
\]
where $x \in C$ is a chosen base point, and $\text{Stab}_{\Gmult^d}(x) ^\circ$ is the identity component of the stabiliser. 
In Section~\ref{ssec:Vogan variety} we saw that if $C \subset V_\lambda$ is any orbit, then it is of the form $C = V_0 \times V_1 \times \dots \times V_n$ where every component $V_i$ is either $\Gmult$ or $\{0\}$. 
Then, if $x = (x_1, \dots, x_n) \in C$ is any point, we have 
\begin{align*}
\text{Stab}_{\Gmult^d}(x) & = \{ g  \in \Gmult^d \mid g \cdot x = x\} \\
& = \{(g_1, \dots, g_d) \in \Gmult^d  \mid (g_1 x_1, \dots, g_d x_d) = (x_1, \dots, x_d) \} 
\end{align*} 
For any $i \in \{1, \dots, d\}$, if $x_i \neq 0$ then $g_i x_i = x_i$ implies $g_i = 1$. On the other hand, if $x_i = 0$, then $g_i$ can be any element of $\Gmult^d $. This implies that the stabiliser of $x$ is of the form 
\[
\text{Stab}_{\Gmult^d }(x) = H_1 \times H_2 \times \dots \times H_d
\]
where $H_i = \Gmult$ or $H_i$ is the trivial group. In any case, this shows that $\text{Stab}_{\Gmult^d }(x)$ is connected, and thus the equivaraint fundamental group  %$\pi_1^{\Gmult^d}(C,x)$ 
and $x$ 
is the trivial group for every orbit $C \subset V_\lambda$. Hence, every orbit has exactly one simple object, up to isomorphism. It remains to show that this isomorphism class is given by the constant sheaf $\mathbbm{1}_C$. But the pullback of any constant sheaf against a smooth surjective map is isomorphic to the constant sheaf for the domain of that map. Since both maps $m, p_1: \Gmult \times C \rightarrow C$ are smooth and surjective, we will always have an isomorphism 
\[
m^* \mathbbm{1}_C \cong p_1^* \mathbbm{1}_C 
\]
Hence, the constant sheaf is always equivariant and we are done. 
\end{proof} 
As we saw in Section~\ref{ssec:Vogan variety}, the orbits the Vogan variety are in bijection with the subsets of simple roots $I \subseteq \dualgroup{S}$, as are the generalised Steinberg representations, by definition (here, we may use the simple roots for either $G$ or $\dualgroup{G}$ since they share the same cardinality). As such, by defining the principal block in this way, we successfully cut out a subcategory of $\Perv_{H_\lambda}(V_\lambda)$ whose simple objects are exactly those perverse sheaves in bijection with generalised Steinberg representations, or, equivalently, those perverse sheaves in bijection with representations of $G(F)$ rather than its pure inner forms.

\begin{example}\label{example:1}
Consider the case where $G = \text{PGL}(2)$, and so $\dualgroup{G} = \SL(2)$. As laid out in \cite{CFMMX}*{Chapter 12}, has two quadratic forms and both are pure. Considering the infinitesimal parameter $\lambda:W_F \rightarrow \dualgroup{G}$ given by 
\[
\lambda(w) = \begin{pmatrix}
|w|^{1/2} & 0 \\
0 & |w|^{-1/2} 
\end{pmatrix}
\]
This parameter produces the Vogan variety given by 
\[
V_\lambda = \left\{\begin{pmatrix} 0 & y \\ 0 & 0 \end{pmatrix} \in \mathfrak{sl}_2 \right\} \cong \mathbb{A}^1 , \qquad H_\lambda  = \left\{ \begin{pmatrix} t & 0 \\ 0 & t^{-1} 
\end{pmatrix} \right\}  \cong \mathbb{G}_m,
\]
though with non-trivial action 
\[
H_\lambda \times V_\lambda \rightarrow V_\lambda; \qquad \begin{pmatrix} t & 0 \\ 0 & t^{-1} 
\end{pmatrix} \cdot \begin{pmatrix} 0 & y \\ 0 & 0 \end{pmatrix} = 
 \begin{pmatrix} 0 & t^2 y \\ 0 & 0 \end{pmatrix}.
\]
This action produces the same orbits $C_0 = \{0\}$ and $C_1 = \Gmult$ as would the usual scaling action, however the quadratic nature of this action produces an equivariant fundamental group
\[
\pi_1^{H_\lambda}(C_1) \cong \ZZ/2\ZZ , 
\]
and trivial equivariant fundamental group on the closed orbit $C_0$. Taken up to isomorphism, the equivariant local systems are $\mathcal{E}$ on $C_1$, the so-called \textit{square-root sheaf}, in addition to the usual trivial local systems on each orbit. The former corresponds to a representation of the non-quasi split form attached to this example, while the latter correspond to representations of $\text{PGL}_2(F)$, where, as usual, $F$ can be any non-Archimedean field. In fact, $\1_{C_0}$ corresponds to the trivial representation of this group while $\1_{C_1}$ corresponds to the Steinberg representation. We can understand the category $\Perv_{H_\lambda}(V_\lambda)$ completely through its stalk table: 
$$\begin{array}{c||c|c}
\mathcal{P} & \mathcal{P}|_{C_0} & \mathcal{P}|_{C_1} \\
\hline\hline
\mathcal{IC}(\mathbf{1}_{C_0}) & \mathbf{1}_{C_0}[0] & 0 \\
\hline
\mathcal{IC}(\mathbf{1}_{C_1}) & \mathbf{1}_{C_0}[1] & \mathbf{1}_{C_1}[1] \\
\hline
\mathcal{IC}(\mathcal{E}) & 0 & \mathcal{E}[1]
\end{array}$$
We can produce the principal block for this category through the isogeny
\[
H_\lambda \to \Gmult ; \qquad t \mapsto t^2,  
\]
and so the action of $H_\lambda$ factors through the usual action by scaling. This produces a block decomposition of the category of $H_\lambda$-equivariant perverse sheaves
\[
\Perv_{H_\lambda}(V_\lambda) = \Perv_{H_\lambda}^\circ(V_\lambda) \oplus \Perv_{H_\lambda}^{-}(V_\lambda),
\]
so that $\Perv_{H_\lambda}^{-}(V_\lambda)$ collects only the perverse sheaf $\IC(\mathcal{E})$. This summand subcategory is thus semisimple, in-fact equivalent to a category of vector spaces, corresponds to the representation of the non-quasisplit form of $\text{PGL}(2)$, while 
\[
\Perv_{H_\lambda}^\circ(V_\lambda) \simeq \Perv_{\Gmult}(\mathbb{A}^1) 
\]
is the principal block for the category and collects $\IC(\1_{C_0})$ and $\IC(\1_{C_1})$, and thus corresponds to representations of $\text{PGL}_2(F)$ itself. This illustrates how the principal block isolates exactly those perverse corresponding to representations of the split form, while representations of pure inner forms are captured by a seperate subcategory. 
\end{example}

\section{Quiver description of the principal block}\label{section:QuiverPrincipalBlock}

To describe the principal block algebraically, we may decompose it into the external tensor of $d$-many copies of a ``base case category", {\it i.e.}, being the category $\Perv_{H_\lambda}^\circ(V_\lambda)$ in the case that $G$ has a unique simple root. The desired algebraic description of the principal block is then achieved by first realising the base case category as the category of modules over the path algebra of a quiver and then using the theory of quivers and path algebras to build the general case for arbitrary $d$. 
Briefly, in this section, we find a $\mathbb{C}$-algebra $E_\lambda$ so that the principal block of the category of $H_\lambda$-equivariant perverse sheaves on $V_\lambda$ are equivalent to the category of finite dimensional modules over $E_\lambda$:
\[
\Perv_{H_\lambda}^\circ(V_\lambda) \simeq \Mod(E_\lambda).
\]
In this case, the principal block has a well known equivalent description in terms of the representations of a bounded quiver. We will employ this description to significant effect, using it to build a general algebraic description of the principal block for generalised Steinberg representations. 
\subsection{The base case category}\label{ssec:toy}
The following well-known result describes the base case of our construction. This base case is exactly the principal block in the example outlined at the end of Section~\ref{section:PS}. We include the proof for completeness and clarity.
\begin{lemma}\label{lemma:quiver}
    Let $\Gmult$ act on the affine line $\mathbb{A}^1$ via multiplication. Then, the category of $\Gmult$-perverse sheaves $\Perv_{\Gmult}(\mathbb{A}^1)$ is equivalent to the category of finite dimensional representations of the quiver 
        \begin{equation}\label{eqn:quiver}
\begin{tikzcd}
	0 && 1
	\arrow["a", curve={height=-12pt}, from=1-1, to=1-3]
	\arrow["b", curve={height=-12pt}, from=1-3, to=1-1]
\end{tikzcd}
\end{equation} 
    subject to the relations $a\circ b= 0$ and $b\circ a = 0$. 
\end{lemma}
\begin{proof}
This lemma is well known. For example, taking the rank 1 case from \cite{LW}*{Theorem~5.4}, we get an equivalence of categories from the representations of the described quiver and the category $\Perv_{\Gmult \times \Gmult}(\mathbb{A}^1)$, where $\Gmult \times \Gmult$ acts via 
\[
(\Gmult \times \Gmult)  \times \mathbb{A}^1 \rightarrow \mathbb{A}^1 ;\quad  (t_1, t_2)\cdot x =  t_1 x t_2^{-1}. 
\] 
Since $x =  t_1 x t_2^{-1} = (t_1 t_2^{-1}) x$, we get a special case of Proposition~\ref{prop:embedding} using the map of algebraic groups
\[
\Gmult \times \Gmult \rightarrow \Gmult ; \quad (t_1, t_2) \mapsto t_1t_2^{-1}.
\]
In this case, the map is surjective and compatible with both actions, but there is a connected kernel (given by the diagonal case where $t_1 = t_2$ in $\Gmult \times \Gmult$). Thus, here Lemma~\ref{lemma:embedding} provides an equivalence
\[
\Perv_{\Gmult }(\mathbb{A}^1) \simeq \Perv_{\Gmult \times \Gmult}(\mathbb{A}^1),
\]
and, hence, $\Perv_{\Gmult }(\mathbb{A}^1) $ is likewise equivalent to the category finite dimensional representations of this quiver. 
\end{proof} 

\begin{definition}\label{definition:nabla}
With reference to the $\Gmult$-orbit stratification of $\mathbb{A}^1 = \mathbb{G}_\text{m} \sqcup \{ 0\}$, set
\[
        \Delta := j_! \mathbbm{1}_{\mathbb{G}_\text{m}}[1],
    \qquad \text{and}\qquad
        \nabla := j_* \mathbbm{1}_{\mathbb{G}_\text{m}}[1],
\]
for $j : \Gmult \hookrightarrow \mathbb{A}^1$. We denote the endomorphism ring of their direct sum as 
\[
E := \End_{\Perv_{\Gmult}(\mathbb{A}^1)}(\Delta \oplus \nabla).
\]
\end{definition}

\begin{remark}
Note that, since $j$ is smooth, there are isomorphisms of functors ${}^pj_! \cong j_![1]$ and ${}^pj_* \cong j_*[1]$, and therefore both $\Delta$ and $\nabla$ are perverse. These objects will provide a complete list of indecomposable projectives for the category in this case, up to equivalence \cite{LW}*{Theorem 2.13}, and will be central in our construction of the path algebra for the general principal block. 
\end{remark}

\begin{proposition}\label{proposition:endomorphism} 
    The objects $\Delta$ and $\nabla$ are the respective projective covers of the simple objects $\IC(\mathbbm{1}_{\mathbb{G}_\text{m}})$ and $\IC(\mathbbm{1}_0)$. Furthermore, the category of equivariant perverse sheaves $\Perv_{\Gmult}(\mathbb{A}^1)$ is equivalent to the category of finitely generated modules over the endomorphism algebra $E$. 
\end{proposition}
\begin{proof}
    Given \cite{LW}*{Theorem 2.13}, it is straightforward to check that $\Delta$ and $\nabla$ correspond to the reducible indecomposable objects in the category of quiver representations, and verifying that they are projective is thus an exercise in linear algebra. 
\end{proof}

\subsection{General case}
For generalised Steinberg representations of $G(F)$ whose set of simple roots $S$ is arbitrarily large, say $|S| = d$, we are now in a position to decompose the corresponding principal block into the external product of $d$ copies of the base case. 
\begin{proposition}\label{proposition:exteriorsheaves}
    There exists an equivalence of categories 
    \[
\Perv_{H_\lambda}^\circ(V_\lambda) \simeq \left(\Perv_{\Gmult}(\mathbb{A}^1)\right)^{\boxtimes d}
    \]
    where $d = |\dualgroup{S}|$ and $ \Perv_{\Gmult}(\mathbb{A}^1)^{\boxtimes d} $ denotes the external tensor of $d$ copies of $\Perv_{\Gmult}(\mathbb{A}^1)$. 
\end{proposition}

\begin{proof}
    We have see that  $\Gmult^d$ acts on $V_\lambda = \mathbb{A}^d$ component-wise, which is equivalent to the product of $d$-copies of $\Gmult$ acting on $d$-copies of $\mathbb{A}^1$.  From \cite{VL}*{Main Theorem}, if $V_0$ is a $G_0$-variety and $V_1$ is a $G_1$ variety, for connected reductive algebraic groups $G_0$ and $G_1$, there is an equivalence of categories 
    \[
\Perv_{G_0 \times G_1}(V_0 \times V_1) \simeq \Perv_{G_0 }(V_0 )\boxtimes \Perv_{G_1 }(V_1).
    \]
    In our case, applying this theorem iteratively obtains
    \[
\Perv_{\Gmult^d}(\mathbb{A}^d) \simeq \Perv_{\Gmult^{\times d}}((\mathbb{A}^1)^{\times d}) \simeq \Perv_{\Gmult}(\mathbb{A}^1)^{\boxtimes d},  
    \]
    which is what we wanted to show.
\end{proof}

\subsection{Simple and projective objects in 
%the category of 
 equivariant perverse sheaves}
\label{ssec:projective}

From the equivalence in Proposition~\ref{proposition:exteriorsheaves}, and its use of \cite{VL}*{Main Theorem}, we see that the image under this equivalence of $\IC(\mathbbm{1}_{C_I})$ and its projective cover in $\Perv_{\Gmult^d}(\mathbb{A}^d)$ can be described as follows.

\begin{definition}\label{definition:ICP}
For any $I\subseteq \dualgroup{S}$, define $\IC_I^i\in \Perv_{\Gmult}(\mathbb{A}^1)$ by
\[
\IC_I^i \ceq
\begin{cases}
\IC(\mathbbm{1}_{0}) & \hat{\alpha}_i \in I  \\
\IC(\mathbbm{1}_{\Gmult}) & \hat{\alpha}_i \notin I,
\end{cases}
\]
and $\IC_I\in \Perv_{\Gmult}(\mathbb{A}^1)^{\boxtimes d}$ by
\[
\IC_I := \IC_I^1 \boxtimes \IC_I^2 \boxtimes \dots \boxtimes \IC_I^d.
%\in \Perv_{\Gmult}(\mathbb{A}^1)^{\boxtimes d}.
\]
Likewise, define $\mathcal{P}_I^i \in \Perv_{\Gmult}(\mathbb{A}^1)$ by
\[
\mathcal{P}_I^i := 
\begin{cases}
\nabla & \hat{\alpha}_i \in I\\
\Delta & \hat{\alpha}_i \notin I.
\end{cases}
\]
and define $\mathcal{P}_I \in \Perv_{\Gmult}(\mathbb{A}^1)^{\boxtimes d}$ by
\[
\mathcal{P}_I := \mathcal{P}_I^1 \boxtimes \mathcal{P}_I^2 \boxtimes \dots \boxtimes \mathcal{P}_I^d.
%\in \Perv_{\Gmult}(\mathbb{A}^1)^{\boxtimes d}. 
\]
\end{definition}

\begin{remark}  
Although matching up $\IC_I$ with with the orbit $C_I$ is not necessary to support our main results, it ensures that the Vogan's conception of the Langlands correspondence can be traced through this equivalence of categories in a natural way. 
Indeed, the category $\Perv_{H_\lambda}(V_\lambda)$ and its equivalent counterparts enjoy many non-trivial autoequivalences that permute isomorphism classes of simple objects in a fashion compatible with the local Langlands correspondence. 
\end{remark}

\begin{proposition}\label{proposition:projectives}
    With reference to Definition~\ref{definition:ICP}, for every $I\subseteq \dualgroup{S}$, the object $\mathcal{P}_I$ of $\Perv_{\Gmult^d}(\mathbb{A}^d)$ is the projective cover of $\IC_I$.
\end{proposition}

\begin{proof}
Combine Proposition~\ref{proposition:endomorphism} with Lemma~\ref{lemma:Bourbaki}.
\end{proof}

\begin{lemma}\label{lemma:Bourbaki}
    Let $A$ and $B$ be a pair of finite dimensional algebras over an algebraically closed field. Then, the following statements hold:
    \begin{enumerate}
        \item simple objects in the category $\Mod(A) \boxtimes \Mod(B)$ are all of the form $S_A \otimes S_B$, where $S_A$ (resp. $S_B$) is a simple object in the category $\Mod(A)$ (resp. $\Mod(B)$);
        \item indecomposable projective objects in the category $\Mod(A) \boxtimes \Mod(B)$ are all of the form $P_A \otimes P_B$, where $P_A$ (resp. $P_B$) is an indecomposable projective in the category $\Mod(A)$ (resp. $\Mod(B)$).
    \end{enumerate}
\end{lemma}

\begin{proof}
Statement (1) follows from Lemma from \cite{BZ76}*{\S 2} which adapts results from \cite{Bourbaki}*{\S 12}. Turning our attention to (2), this proof is straightforward, though we could not find it explicitly in the literature so we provide one here for completeness: from \cite{Deligne}*{Proposition~5.4}, there is an equivalence of categories 
\[
\Mod(A) \boxtimes \Mod(B) \simeq \Mod(A \otimes B)
\]
Taking the obvious morphism of algebras 
\[
f: A \rightarrow A \otimes B ; \quad x \mapsto x \otimes 1_B,
\]
if $P_A \in \Mod(A)$ is an indecomposable projective, we extend by scalars to get the $A\otimes B$-module $P_A \otimes B $, written
\[
f_!(P_A) := P_A \otimes B 
\]
From this module, we can use that $f_!$ has left adjoint $f^*$ (restriction of scalars) to get a natural isomorphism of functors valued in vector spaces
\[
\Hom_{A\otimes B}(f_!(P_A), - ) \cong \Hom_A(P_A , f^*(-)) 
\]
Which are both exact, since $P_A$ is projective and restriction of scalars $f^*$ is exact. This implies that $f_!(P_A)$ is projective. Without loss of generality, we can assume that $B$ is a basic algebra\cite{ASS}*{Corollary~1.6.10}, which means that there is a decomposition of $B$ as a $B$-module given by
\[
B \cong \bigoplus_{i} P_B(i) 
\]
where $i$ runs over the set of isomorphism classes of simple objects $S_B(i) \in \Mod(B)$, so that $P_B(i) \rightarrow S_B(i)$ is a projective cover, for every $i$. Hence, we get a decomposition of $A\otimes B$ modules 
\[
P_A \otimes B \cong P_A \otimes \left( \bigoplus_{i} P_B(i) \right) 
\]
and since the tensor product respects direct sums, this can be rewritten as 
\[
P_A \otimes B \cong \bigoplus_i P_A \otimes P_B(i). 
\]
Since any summand of a projective module is projective, $P_A \otimes P_B(i)$ is projective for every $i$. Repeating this process for every projective $P_A \in \Mod(A)$ shows that any $A\otimes B$ module of the form $P_A \otimes P_B$ is projective. Furthermore, the set isomorphism classes of indecomposable projectives is always in bijection with the set of isomorphism classes of simple objects in any category of modules over a finite-dimensional algebra \cite{ASS}*{\S~III.2}. As such, the indecomposable projectives in $\Mod(A\otimes B)$ of this form exhaust the list of indecomposable projectves, up to isomorphism, finishing the proof. 
\end{proof}

\begin{proposition}\label{prop:endalg}
    Category $\Perv_{\Gmult}(\mathbb{A}^1)^{\boxtimes d}$ is equivalent to the category of finitely generated modules over the endomorphism algebra 
    \[
\text{End}_{\Perv_{\Gmult}(\mathbb{A}^1)^{\boxtimes d}}(\mathcal{P}) = \bigoplus_{I,J\subseteq \dualgroup{S}} \text{Hom}_{\Perv_{\Gmult}(\mathbb{A}^1)^{\boxtimes d}}(\mathcal{P}_I, \mathcal{P}_J) , 
\]
where $i$ and $j$ run over all isomorphism classes of indecomposable projectives in the category $\Perv_{\Gmult}(\mathbb{A}^1)^{\boxtimes d} $, and the $\mathcal{P}_I$ and $\mathcal{P}_J$ are choices of representatives from each appropriate class, with $\mathcal{P}$ the direct sum of all such choices. 
\end{proposition}

\begin{proof}
  Proposition~\ref{proposition:endomorphism}, Following \cite{Deligne}*{Proposition 2.14}, the external tensor of finitely many categories of finitely generated modules over finite-dimensional algebras is itself equivalent to a category of modules over a finite-dimensional algebra. 
  From Proposition~\ref{proposition:endomorphism}, each tensored component of $\Perv_{\Gmult}(\mathbb{A}^1)^{\boxtimes d}$ is equivalent to the category of finitely generated modules over the algebra $A$. 
  It follows from \cite{ASS}*{Lemma~I.6.5} that any finite-dimensional algebra is Morita equivalent to the endomorphism algebra of its progenerator, which can be taken to be direct sum of projective covers of simple objects (one from each isomorphism-class) in its category of finitely-generated modules. 
  Item (2) from Proposition~\ref{lemma:Bourbaki} shows that taking $\mathcal{P}_I$ for each $I \subseteq \dualgroup{S}$ gives a complete list of these projectives for $\Perv_{\Gmult}(\mathbb{A}^1)^{\boxtimes d}$.
\end{proof}

\begin{definition}\label{def:Alambda}
%To ease notation below, set
%\[
%\Perv := \Perv_{\Gmult}(\mathbb{A}^1)^{\boxtimes d}
%\]
With reference to Proposition~\ref{prop:endalg}, set
\[
E_\lambda := \text{End}_{\Perv_{\Gmult}(\mathbb{A}^1)^{\boxtimes d}}(\mathcal{P}).
%= \text{End}_{\Perv}(\mathcal{P}).
\]

\end{definition} 

\begin{corollary}\label{cor:firstequivalence}
    There is an equivalence of categories 
    \[
\Perv_{H_\lambda}^\circ(V_\lambda) \simeq \Mod(E_\lambda) 
    \]
\end{corollary}
\begin{proof}
    This follows from Proposition~\ref{prop:endalg}, given the equivalence stated in Proposition~\ref{proposition:exteriorsheaves}.
\end{proof}
The description of the principal block as a category of modules over $E_\lambda$ will be the intermediate step in our comparison between this category and the extensions between generalised Steinberg representations as our strategy is write down an isomorphism between $E_\lambda$ and the extension algebra produced by these representations. To do this, we will compare the graded components of each algebra piece by piece.  

\subsection{Grading}\label{subsection:grading}

Since $E_\lambda$ is a basic, finite-dimensional algebra, it is the path algebra of a quiver, and therefore it carries a natural grading according to path length \cite{ASS}*{\S II.1}. Instead of taking this grading naively, we reconstruct the grading by hand using morphisms between the projective objects of $\Perv_{\Gmult}(\mathbb{A}^1)^{\boxtimes d}$ that sum to a projective generator. This reconstruction is necessary to introduce several algebraic subtleties which will afford us an explicit description of the isomorphism between $E_\lambda$ and the extension algebra on the direct sum of generalised Steinberg representations.

\begin{lemma}\label{lemma:dimhom}
    For any subsets $I,J \subseteq \dualgroup{S}$: 
    \begin{enumerate}
\item 
    $\dim \Hom_{\Perv_{\Gmult}(\mathbb{A}^1)^{\boxtimes d}}(\mathcal{P}_I, \mathcal{P}_J)  = 1$; 
\item 
    any non-trivial morphism $f \in \Hom_{\Perv_{\Gmult}(\mathbb{A}^1)^{\boxtimes d}}(\mathcal{P}_I, \mathcal{P}_J)$ is equivalent to a morphism of the form 
        \[
\tilde{f} = f_1 \otimes \dots \otimes f_{d} 
        \]
        where $f_i \in \Hom_E\left(\mathcal{P}_I^i, \mathcal{P}_J^i\right)$ for every $i\in \{ 1, 2, \ldots , d\}$. 
    \end{enumerate}
\end{lemma}

\begin{proof}
By \cite{Deligne}*{Corollary~5.4}, and recalling $\mathcal{P}_I = \mathop{\boxtimes}\limits_{i=1}^{d}\mathcal{P}_I^i$ and $\mathcal{P}_J = \mathop{\boxtimes}\limits_{i=1}^{d}\mathcal{P}_J^i$ from Definition~\ref{definition:ICP}, we get an isomorphism of homsets
    \begin{align*}
\Hom_{\Perv_{\Gmult}(\mathbb{A}^1)^{\boxtimes d}}\left( \mathcal{P}_I, \mathcal{P}_J\right) 
& = \Hom_{\Perv_{\Gmult}(\mathbb{A}^1)^{\boxtimes d}}\left(\mathop{\boxtimes}\limits_{i=1}^{d} \mathcal{P}_I^i, \mathop{\boxtimes}\limits_{i=1}^{d} \mathcal{P}_J^i\right)\\
&\cong \mathop{\boxtimes}\limits_{i=1}^{d} \Hom_{E}\left(\mathcal{P}_I^i, \mathcal{P}_J^i\right)
    \end{align*} 
    From this isomorphism, claim $(2)$ follow immediately, and claim $(1)$ follows from the fact that $\dim \Hom_A\left(\mathcal{P}_I^i, \mathcal{P}_J^i\right) = 1$ for every $i$, which we saw in the proof of Proposition~\ref{proposition:endomorphism}. 
\end{proof}

\begin{definition}\label{definition:grading} 
   Using Lemma~\ref{lemma:dimhom}, we see that any non-trivial morphism $f$ in $\Hom_{\Perv}(\mathcal{P}_I, \mathcal{P}_J)$ is equivalent to a morphism that may be written as 
   \[
\tilde{f} = f_{}^1 \otimes f_{}^2 \otimes \dots \otimes f_{}^d. 
   \]
   Given this description, define the \textbf{degree} of $f$ as 
   \[
\text{deg}(f_{}) := \char"0023\{ f^i: \mathcal{P}_I^i \rightarrow \mathcal{P}_J^i \mid f^i \not\cong \text{id}_{\mathcal{P}_I^i}\}.
   \]
   if $f\neq 0$, and set $\deg(f) = 0$ if $f = 0$. 
\end{definition}

\begin{lemma}\label{lemma:degree}
    For $f \in \Hom_{\Perv_{\Gmult}(\mathbb{A}^1)^{\boxtimes d}}(\mathcal{P}_I, \mathcal{P}_J)$, if $ f $ is non-zero then \[
    \deg(f ) = \delta(I,J),
    \]
    the cardinality of the symmetric difference $I \ominus J$. 
\end{lemma}

\begin{proof}
    If $f^i = \text{id}_{\mathcal{P}_I^i}$ for every $i$, the result is clear. Otherwise, for every $i$ such that $f^i \neq \text{id}_{\mathcal{P}_I^i}$, it follows that the domain and codomain of $f^i$ are not isomorphic. This is to say that either $\hat{\alpha}_i \in  I$ and $\hat{\alpha}_i \notin J$, or visa versa. But this means that $\hat{\alpha}_i \in  I \cup J$ but $\hat{\alpha}_i \notin I \cap J$, or equivalently $\hat{\alpha}_i \in  I \ominus K$, and so $\deg(f) \leq \delta(I,J)$. Conversely, if $\hat{\alpha}_i \in  I \ominus J$, then domain and codomain of $f^i$ are non-isomorphic, by definition, implying $\delta(I,J) \leq \deg(f )$. 
\end{proof}

This grading is essential for two reason: for one, it is a precedent to understand $E_\lambda$ as a Koszul algebra, the content of Section~\ref{ssec:koszulalgebra}, and secondly, as previously stated, it provides the bridge to the extension algebra for generalised Steinberg representations, whose cohomological grading is likewise governed by the symmetric different of subsets of simple roots. 

\begin{lemma}\label{lemma:settheory}
    For subsets $I,J, K \subseteq \dualgroup{S}$,  
    \[
I \ominus K \subseteq (I \ominus J)  \cup (J \ominus K)  
    \]
\end{lemma}

\begin{proof}
    Suppose $\hat{\alpha}_i \in  I \ominus K$, and $\hat{\alpha}_i \in  I$ in particular. Then, $\hat{\alpha}_i \notin K$. If $\hat{\alpha}_i \in  J$, then $\hat{\alpha}_i \in  J \ominus K$. If $\hat{\alpha}_i \notin J$, then $\hat{\alpha}_i \in  I \ominus J $. Either way, $\hat{\alpha}_i \in ( I \ominus J ) \cup (J \ominus K)$. 
\end{proof}

\begin{proposition}\label{proposition:hompairing}
    For subset $I,J,K \subseteq S$, 
    the pairing
    \[
\Hom_{\Perv_{\Gmult}(\mathbb{A}^1)^{\boxtimes d}}^{\delta(I,J)}(\mathcal{P}_I, \mathcal{P}_J) \otimes \Hom_{\Perv_{\Gmult}(\mathbb{A}^1)^{\boxtimes d}}^{\delta(J,K)}(\mathcal{P}_J, \mathcal{P}_K) \rightarrow \Hom_{\Perv_{\Gmult}(\mathbb{A}^1)^{\boxtimes d}}^{\delta(I,K)}(\mathcal{P}_I, \mathcal{P}_K)
    \]
   given by $(f_{IJ} , f_{JK} ) \mapsto f_{JK} \circ f_{IJ}$ is an isomorphism if and only if $\delta(I,J) + \delta(J, K) = \delta(I,K)$. Otherwise, the pairing is the zero map. 
\end{proposition}

\begin{proof}
We abuse the notation introduced in Lemma~\ref{lemma:quiver}, so that $a: \Delta \rightarrow \nabla$ and $b: \nabla \rightarrow \Delta$ are the unique, non-trivial morphisms between the given indecomposable projectives in $ \Perv_{\Gmult}(\mathbb{A}_1)$. 
 By Lemma~\ref{lemma:dimhom}, all homsets between indecomposable projectives in $\Perv_{\Gmult}(\mathbb{A}^1)^{\boxtimes d}$ are one dimensional, and so take $ f_{IJ} \in \Hom_{\Perv_{\Gmult}(\mathbb{A}^1)^{\boxtimes d}}^{\delta(I,J)}(\mathcal{P}_I, \mathcal{P}_J)$ and $f_{JK} \in \Hom_{\Perv_{\Gmult}(\mathbb{A}^1)^{\boxtimes d}}^{\delta(J,K)}(\mathcal{P}_J, \mathcal{P}_K)$ be generators. 
 In the first direction, assume that the pairing is an isomorphism. Then $f_{JK} \circ f_{JK} \neq 0$ and this map can be taken as a generator for $\Hom_{\Perv_{\Gmult}(\mathbb{A}^1)^{\boxtimes d}}^{\delta(I,K)}(\mathcal{P}_I, \mathcal{P}_K)$. If this is the case, then every tensored component
    \[
(f_{JK} \circ f_{IJ} )^i = f_{JK}^i \circ f_{IJ}^i
    \]
    likewise nonzero. To prove this direction, it suffices to show that 
    \[
(I \ominus J)  \cup (J \ominus K) = I \ominus K
    \]
whenever this tensored component $ f_{JK}^i \circ f_{IJ}^i$ is non-zero. First, suppose $\hat{\alpha}_i \in  (I\ominus J) \cup (J \ominus K)$. There are three cases:
    \begin{enumerate}
        \item $(\hat{\alpha}_i \in  I)$: Then, $\hat{\alpha}_i \in  I\ominus J$ and so $\hat{\alpha}_i \notin J$. This implies that $\mathcal{P}_I^i = \nabla$ and $\mathcal{P}_J^i = \Delta$ and so $f_{IJ}^i = b$. Since $ f_{JK}^i \circ f_{IJ}^i \neq 0$, it follows that $ f_{JK}^i \circ f_{IJ}^i = b$ as well, and so $f_{JK}^i = \text{id}_\Delta$, from which it follows that $\hat{\alpha}_i \notin K$ and thus $\hat{\alpha}_i \in  I\ominus K$. 
        \item $(\hat{\alpha}_i \in  J):$ In this case, $\hat{\alpha}_i \notin I$, so $f_{IJ}^i = a$. Similar to the above, this implies that $ f_{JK}^i \circ f_{IJ}^i = a$ and so $f_{JK}^i = \text{id}_\Delta$. Thus $\hat{\alpha}_i \in  K$ and, since $\hat{\alpha}_i \notin I$, it follows that $\hat{\alpha}_i \in  I\ominus K $. 
        \item $(\hat{\alpha}_i \in  K)$: Then, $\hat{\alpha}_i \notin J$, implying that $f_{JK} = b$. Since $ f_{JK}^i \circ f_{IJ}^i \neq 0$, it can only be the case that $f_{IJ}^i= \text{id}_\Delta$, and so $\hat{\alpha}_i \notin I$. Thus, $\hat{\alpha}_i \in  I \ominus K$. 
    \end{enumerate}
    
    This implies that $(I\ominus J) \cup (J\ominus K) \subseteq I \ominus K$, and the opposite inclusion was proved in Lemma~\ref{lemma:settheory}.  
    
   Conversely, suppose that $\delta(I,J) + \delta(J,K) = \delta(I,K)$. From Lemma~\ref{lemma:settheory}, it must be the case that 
    \begin{equation}\label{equation:equalsets}
(I \ominus J)  \cup (J \ominus K) = I \ominus K 
    \end{equation} 
    Consider $f_{JK} \circ f_{IJ} $. 
    For the purpose of deriving a contradiction, assume there exists an $i$ so that $ f_{JK}^i \circ f_{IJ}^i=0$. 
    Then, $f_{IJ}^i \neq \text{id}_{\mathcal{P}_I^i}$ and $f_{JK}^i \neq \text{id}_{\mathcal{P}_J^i}$ ({\emph{i.e.}}, either $f_{JK} \circ f_{IJ} = a\circ b$ or $ f_{JK}^i \circ f_{IJ}^i = b\circ a$. 
    But then $ \hat{\alpha}_i \in  I \ominus J $ and $\hat{\alpha}_i \in  J \ominus K $. If $\hat{\alpha}_i \in  I$, then $ \hat{\alpha}_i \notin J$, implying that $\hat{\alpha}_i \in  K$. 
    On the other hand, if $\hat{\alpha}_i \notin I$, then $\hat{\alpha}_i \in  J$, and so $\hat{\alpha}_i \notin K$. 
    In both cases, $\hat{\alpha}_i \notin I \ominus K$, contradicting Equation~\ref{equation:equalsets}. 
    It must then be the case that $ f_{JK}^i \circ f_{IJ}^i\neq 0$ for every $i$ and we get $ f_{JK} \circ f_{IJ} \neq 0$. 

    Finally, suppose that $\delta(I,J) + \delta(J,K) \neq \delta(I,K)$. 
    Again referring to Lemma~\ref{lemma:settheory}, this can only happen if there exists some $\hat{\alpha}_i \in  (I \ominus J)  \cup (J \ominus K)$ so that $\hat{\alpha}_i \notin I \ominus K$ by Lemma~\ref{lemma:degree}. 
    It follows that $\hat{\alpha}_i \in  I$ and $\hat{\alpha}_i \in  K$, but $\hat{\alpha}_i \notin J$, which, by definition, implies that $f_{IJ}^i = b$ and $f_{JK}^i = a$. Hence,
    \[
 f_{JK}^i \circ f_{IJ}^i = a \circ b = 0.
    \]
    Since one of the tensored components of $ f_{JK} \circ f_{IJ}$ is zero, the entire morphism must also be the zero map and we are done. 
\end{proof}

\begin{corollary}
Definition~\ref{definition:grading} defines a grading on the algebra $E_\lambda$. 
\end{corollary}

\begin{proof}
Any morphism $f \in E_\lambda := \End_{\Perv_{\Gmult}(\mathbb{A}^1)^{\boxtimes d}}(\mathcal{P} )$ is the direct sum of a morphism of the form $f_{IJ} \in \Hom_{\Perv_{\Gmult}(\mathbb{A}^1)^{\boxtimes d}}(\mathcal{P}_I, \mathcal{P}_J)$, for $I,J \subseteq S$. 
Moreover, from Proposition~\ref{proposition:hompairing}, morphisms $f_{IJ}: \mathcal{P}_I \rightarrow \mathcal{P}_J$ and $f_{JK} :\mathcal{P}_{J} \rightarrow \mathcal{P}_J$ can be composed so that either 
\[
\deg(f_{IJ}) + \deg(f_{IK}) = \deg(f_{JK} \circ f_{IJ} ) 
\]
or $ f_{JK} \circ f_{IJ} = 0$. Either way, the composition occurs in the appropriate graded piece of the algebra. 
\end{proof}

Our ultimate goal is to produce a Koszul duality between extensions of gnerelised Steinberg representations of $G(F)$ and the extensions of perverse sheaves that serve as their geometric counterpart under the Langlands correspondence. In this regard, we briefly characterise Koszul duality, generally, and demonstrate that $E_\lambda$ is Koszul, as this algebra will serve as a bridge between either side of the Langlands correspondence. 

\subsection{Koszul Algebras and Koszul duality}\label{ssec:koszulalgebra} 

Recall that a \textit{Koszul algebra} is a positively graded algebra $A = \bigoplus_{i \geq 0} A_i$ where the degree zero component $A_0$ has a linear projective resolution. Koszul algebras are quadratic, meaning as a path algebra their relations are generated by a subset of degree two components. Such algebras enjoy a natural duality, which may be demonstrated using a convenient combinatorial gadget in the case that the algebra is the bounded path algebra of a quiver: given $A = \mathbb{C}[Q]/I$ for a finite quiver $Q$, with relations generated by elements taken from the subspace $ \CC[Q]_2 \subset \CC[Q]$ of paths of length two, following \cite{RMV}*{\S 1}, define the bilinear form 
\[
\langle \cdot , \cdot \rangle : \CC[Q^{\op}]_2 \times \CC[Q]_2\rightarrow \CC \quad ;  \quad \langle \gamma^{\op}, \gamma' \rangle = 
\begin{cases}
    0 \text{ if } \gamma \neq \gamma'  \\
    1 \text{ otherwise.} 
\end{cases}
\]
We can use this form to produce an ideal $I^\perp \lhd \CC[Q^{\text{op}}]$ 
\[
I^\perp := \langle \gamma \in \CC[Q^{\op}]_2 \mid \langle \gamma, \gamma' \rangle = 0, \ \forall \gamma' \in I\rangle. 
\]
Morally, the ideal $I^\perp$ can be thought of as the ideal generated by complement of the quadratic relations generating $I$ in the degree 2 graded piece of the algebra, and indeed this description is precise in many cases. Clearly, if we run this process twice, it will recover the algebra with which we started, providing a duality, and the algebra 
$$A^! := \CC[Q^{\text{op}}]/I^\perp$$
is appropriately called the $\textit{quadratic dual}$. \newline 

In the case that $\mathbb{C}[Q]/I$ is Koszul, the linear projective resolution of its degree zero component produces an isomorphism 
\begin{equation}\label{eqn:quadraticKoszul}
    A^! \cong \Ext_{A}^\bullet(A_0, A_0),
\end{equation}
and we refer to $A^!$ as the \textit{Koszul dual} of $A$ in this case. In fact, the condition that a quadratic algebra $A$ have a dual that satisfies the above isomorphism is equivalent to the condition that $A$ be Koszul \cite{PP}*{Definition 2.1.1}.  Note that considering $A_0$ as an $A$-module is equivalent to considering the $A$-module produced by selecting a representative from each isomorphism class of simple objects and taking their direct sum, and it is this perspective that we will consider throughout. Koszul algebras enjoy several convenient properties. For our purposes, we will use the fact, if $A$ and $B$ are Koszul, then $A \otimes B$ is Koszul \cite{PP}*{Corollary 3.1.2}. In fact, 
\[
(A \otimes B)^! \cong A^! \overline{\otimes} B^!, 
\]
where the notation $\overline{\otimes}$ refers to the alternating tensor product \cite{PP}*{Proposition 3.1.1}. 

Our strategy in this paper is to realise $\Perv_{H_\lambda}^\circ(V_\lambda)$ for arbitrary split rank $d \geq 1$ as the product of $d$ copies of the split rank $1$ case, which, as we have seen, is equivalent to the category of modules over the bounded path algebra $\mathbb{C}[Q]/\langle ab, ba \rangle$ for the quiver $Q$, described by Equation~\ref{eqn:quiver}. Of course, $E$, appearing in Definition~\ref{definition:nabla}, is the endomorphism algebra for the projective objects in a category equivalent to modules over this bounded path algebra. Hence this bounded path algebra $\mathbb{C}[Q]/\langle ab, ba \rangle$ and $E$ are Morita equivalent, and, in fact, isomorphic. 
\begin{lemma}
    For the quiver $Q$ from Equation~\ref{eqn:quiver}, the bounded path algebra $\CC[Q]/ \langle ab, ba \rangle$, and hence the algebra $E$ from Definition~\ref{definition:nabla}, is Koszul and there is an isomorphism
    \[
\left( \CC[Q]/\langle ab, ba \rangle \right)^! \cong \CC[Q] 
    \]
\end{lemma}
\begin{proof}
   Clearly this algebra is quadratic and there is an isomorphism $Q^{\text{op}} \cong Q$ which respects the relations $\langle ab, ba \rangle$. Since these relations span the subspace of length 2 paths, it is straightforward to calculate that 
\[
\langle ab, ba \rangle^\perp = \langle 0 \rangle, 
\]
and so $A^!$ is given by the path algebra for the same quiver but without any bounding relations. Moreover, we can write down a precise projective resolution for the simple object at node at $0$, as 
\[\begin{tikzcd}[ampersand replacement=\&]
	\dots \& \CC \& \CC \& \CC \& \CC \& \CC \\
	\dots \& \CC \& \CC \& \CC \& \CC \& 0
	\arrow[from=1-1, to=1-2]
	\arrow["0", from=1-2, to=1-3]
	\arrow["0", curve={height=-6pt}, from=1-2, to=2-2]
	\arrow["1", from=1-3, to=1-4]
	\arrow["1", curve={height=-6pt}, from=1-3, to=2-3]
	\arrow["0", from=1-4, to=1-5]
	\arrow["0", curve={height=-6pt}, from=1-4, to=2-4]
	\arrow["1", two heads, from=1-5, to=1-6]
	\arrow["1", curve={height=-6pt}, from=1-5, to=2-5]
	\arrow[curve={height=-6pt}, from=1-6, to=2-6]
	\arrow[from=2-1, to=2-2]
	\arrow["1", curve={height=-6pt}, from=2-2, to=1-2]
	\arrow["1"', from=2-2, to=2-3]
	\arrow["0", curve={height=-6pt}, from=2-3, to=1-3]
	\arrow["0"', from=2-3, to=2-4]
	\arrow["1", curve={height=-6pt}, from=2-4, to=1-4]
	\arrow["1"', from=2-4, to=2-5]
	\arrow["0", curve={height=-6pt}, from=2-5, to=1-5]
	\arrow["0"', two heads, from=2-5, to=2-6]
	\arrow[curve={height=-6pt}, from=2-6, to=1-6]
\end{tikzcd}\]
with the obvious dual resolution for the node at 1. Thus, both resolutions alternate between the pair of projective objects in the category, and thus the category enjoys extensions of arbitrary length, with extensions between simple objects in a common isomorphism class concentrated in even degrees and extensions between objects in different classes concentrated in odd degrees. Thus, even extension correspond to successive cycles around the quiver $Q$, while odd paths correspond to the same cycles plus an additional length one path, which one may use to write down an isomorphism
\[
\CC[Q] \cong \Ext_A^\bullet(A_0 , A_0), 
\]
and since the former is the quadratic dual, both $\CC[Q]/\langle ab, ba \rangle$ and its dual are Koszul. 
\end{proof}
It is known to experts that this is a precise description of the endomorphism algebra of simple perverse sheaves in the $\Gmult$-equivariant derived category for $\mathbb{A}^1$, and a general description of this phenomenon will be the content of our final section. Note that, as the quivers we encounter are all built $Q$ above, they are all isomorphic to their opposite. Hence, we will often conflate $Q$ and $Q^{\text{op}}$ \\

We end this section by briefly describing how we may think of successive tensor products of the path algebra $A =\CC[Q]/\langle ab, ba\rangle$. One can use the relations for tensor products of finite dimensional algebras from \cite{ZL}*{Lemma 1.3} to see that $A \otimes A$ is equivalent to the path algebra of the quiver 
\[\begin{tikzcd}[ampersand replacement=\&]
	\& 11 \\
	01 \&\& 10 \\
	\& 00
	\arrow[curve={height=-6pt}, from=1-2, to=2-1]
	\arrow[curve={height=6pt}, from=1-2, to=2-3]
	\arrow[curve={height=-6pt}, from=2-1, to=1-2]
	\arrow[curve={height=6pt}, from=2-1, to=3-2]
	\arrow[curve={height=6pt}, from=2-3, to=1-2]
	\arrow[curve={height=-6pt}, from=2-3, to=3-2]
	\arrow[curve={height=6pt}, from=3-2, to=2-1]
	\arrow[curve={height=-6pt}, from=3-2, to=2-3]
\end{tikzcd}\]
with all cycles and paths of length $\ell > 2$ set to zero, and with all ``diagrams commuting", \emph{i.e.}, relations associating any two paths of equal length with the same source an target. Generally, the algebra $E^{\otimes d}$ will be equivalent the the path algebra for the double quiver formed by the graph associated to a dimension $d$ hypercube, with the analogous relations. For our purposes, however, we will need a precise presentation of this quiver according to tensor products of morphisms between the projective objects in $\Perv_{\Gmult}(\mathbb{A}^1)$. 
\section{The quiver of the extension algebra of Steinberg representations}\label{section:quiver}
\label{section:quiveralgebra}
Our strategy in this section is to realise the extension algebra produced by extensions between generalised Steinberg representations in $\Rep(G)$ as the path algebra of a quiver. As described in Proposition~\ref{prop:dat}, the structure of this algebra is governed by the combinatorics of subsets of simple roots of $G$. Since the generalised Steinberg representations are in bijection with these subsets, the powerset lattice of simple roots provides an obvious choice to build this presentation. 

Let $\mathscr{P}(S)$ be the quiver given by the powerset of simple roots $S$, whose set of vertices $\mathscr{P}(S)_0$ is given by subsets $I \subseteq S$ and whose set of arrows $\mathscr{P}(S)_1$ is given by inclusions of subsets $I \subseteq J$ where $|J \setminus I| = 1$ and use 
\[
\gamma_{IJ}:I \rightarrow J
\]
to denote the arrow corresponding to this inclusion. The stationary path at any vertex $I$, {\it i.e.}, the paths of length zero, will be denoted by $\varepsilon_I$. 
Let $Q$ be the double quiver generated by $\mathscr{P}(S)$, so that 
\[
Q_0 = \mathscr{P}(S)_0; \quad Q_1 = \mathscr{P}(S)_1 \cup \mathscr{P}(S)_1^{\op}  
\]
where $\mathscr{P}(S)_1^{\op}$ is a second copy of all arrows from the original quiver, but with their directions reversed. We denote these arrows by $\gamma_{JI}:J \rightarrow I$ whenever $\gamma_{IJ}:I \rightarrow J$ is an arrow in $\mathscr{P}(S)_1$. Let $\text{Path}(Q)$ denote the set of paths in $Q$ and define the functions
\begin{align*}
    s:\text{Path}(Q) \rightarrow Q_0; \qquad & (\gamma:I \rightarrow J) \mapsto I \\
    t:\text{Path}(Q) \rightarrow Q_0; \qquad & (\gamma:I \rightarrow J) \mapsto J \\
    \ell: \text{Path}(Q) \rightarrow \mathbb{N};\ \, \qquad & (\gamma:I \rightarrow J) \mapsto \char"0023\{ \varphi \in Q_1 \mid \varphi \in \gamma \}, 
\end{align*}
whereby $\varphi \in \gamma$ we mean $\varphi$ appears in as an arrow in the path $\gamma$, counting multiplicity, and we refer to $\ell(\varphi)$ as the \emph{length} of the path. We denote by $\CC[Q]$ the path algebra of $Q$ (\emph{c.f.}, \cite{ASS}*{Definition 1.2}). By abuse of notation we view $\gamma_{IJ}$ simultaneously as an arrow in the quiver $Q$ as well as a basis element of $\CC[Q]$, and generalise the notation so that $\gamma_{IJ}:I\rightarrow J$ may denote a path from $I$ to $J$, in the case where $|I\setminus J| \neq 1$ and $|J \setminus I|\neq 1$. 
In the same way, we write $\varepsilon_I$ to denote the stationary path at $I$ as well as the associated primitive idempotents in $\CC[Q]$. 

\begin{definition}
Recall the definition of the generalised Steinberg representation $\sigma_I$ of $G$ corresponding to the subset $I \subseteq S$ of simple roots. Then, write 
\[
\Sigma_\lambda := \bigoplus_{I \subseteq S} \sigma_I 
\]
where $\lambda:W_F' \rightarrow {}^L G$ is the infinitesimal parameter that collects these representations \textemdash that is, $\lambda$ is the infinitesimal parameter of the trivial representation $\1_{G}$ of $G(F)$, as in Section~\ref{subsection:parameters}.

\end{definition}

\begin{lemma}\label{lemma:quiversurjection}
    There is a surjection of associative algebras 
    \[
    \varphi: \CC[Q] \twoheadrightarrow \Ext_G^\bullet(\Sigma_\lambda, \Sigma_\lambda),
    \]
    and hence an isomorphism of algebras
    \[
\CC[Q]/\ker(\varphi) \cong \Ext_G^\bullet(\Sigma_\lambda, \Sigma_\lambda).
    \]
\end{lemma}

\begin{proof}
   From the above definition, $\varepsilon_I$ is the idempotent element in $\CC[Q]$ corresponding to the stationary path attached to $I\subseteq S$. Let $\varphi(\varepsilon_I)$ be the identity morphism in  $\Hom_G(\sigma_I, \sigma_I)$. Recall that the arrows in $\Gamma$ are given by subsets $I,J \subseteq S$ where $|I \setminus J| = 1$ or $|J \setminus I| = 1$. Define 
    \[
\varphi(\gamma_{IJ}) \in \Ext_G^1(\sigma_I, \sigma_J)
    \]
    as the unique representative of the group whose corresponding extension appears as a factor in the degree $d = |S|$ extension between $\mathbbm{1}_G$ and $\St_G$ that is given by the Aubert-Deligne-Lusztig sequence for $\mathbbm{1}_G$, or its Aubert dual, depending on the orientation of the extension. Here, the Aubert-Deligne-Lusztig sequence is defined as in \cite{NP}*{\S1}. Write the degree-$1$ extension $\varphi(\gamma_{IJ})$ as 
    \[
0 \rightarrow \sigma_J \rightarrow E_{IJ}  \rightarrow \sigma_I \rightarrow 0.
    \]
    Applying the cohomological functor $R\Hom_G( \sigma_I, -)$ to this sequence produces the usual long-exact sequence in cohomology. Since, up to equivalence, only one extension between $\sigma_I$ and $\sigma_J$ exists and is in degree 1, the connecting homomorphism 
    \[
    \kappa:\Hom_G(\sigma_I, \sigma_I) \rightarrow \Ext_G^1(\sigma_I, \sigma_J)
    \]
    is an isomorphism, and we have
    \[
    \varphi(\gamma_{IJ}) = \kappa(1_{\sigma_{I}}). 
    \]
by construction. The multiplicative structure on $\Ext_G^\bullet(\Sigma_\lambda, \Sigma_\lambda)$ induced by the cup product thus satisfies the conditions necessary to apply the universal property of path algebras from \cite{ASS}*{Theorem 1.8}. The universal property says that the assignment $\varphi$ induces a surjective  morphism of algebras
    \[
\varphi: \CC[Q] \twoheadrightarrow \Ext_G^\bullet(\Sigma_\lambda, \Sigma_\lambda),
    \]
    and this map respects the grading of both algebras, by construction, providing the desired surjection of associative algebras. 
\end{proof}

The goal of this section is to describe the algebra $\Ext_G^\bullet(\Sigma_\lambda, \Sigma_\lambda)$ explicitly. Given Lemma~\ref{lemma:quiversurjection}, the problem is reduced to understanding the kernel of $\varphi$. To do this, we begin by identifying a set of paths in $Q$ in order to characterise the non-trivial extensions in $\Ext_G^\bullet(\Sigma_\lambda, \Sigma_\lambda)$.

\begin{definition}\label{definition:kernelset} 
    Let $\text{\normalfont Path}(Q)$ be the set of all paths in $Q$, now interpreted simultaneously as the set of paths in $Q$ as well as a $\CC$-basis for $\CC[Q]$, and let $\Gamma \subset \text{\normalfont Path}(Q)$ be the set of paths in $Q$ of the form 
    \[
 \gamma = \gamma_{I_{m-1} I_{m}}  \gamma_{I_{m-2} I_{m-1}} \dots \gamma_{I_0} \gamma_{I_{{1}}}:I_0 \rightarrow I_m
    \]
    satisfying $m = \delta(I_0, I_{m})$, where, again, $\delta(I_{i-1}, I_{i}) = 1$ for all $1 \leq i \leq m$. 
\end{definition}
The set $\Gamma$ is fundamental to our analysis: it consists precisely of those paths of minimal length between any pair of vertices $I,J \subseteq S$. In the powerset lattice, this minimal length is exactly the symmetric difference and thus coincides precisely with a non-trivial extension between the generalised Steinberg representations associated with these subsets, as we saw in Proposition~\ref{prop:dat}. We shall see that the presentation of the extension algebra of generalised Steinberg representations follows easily once this correspondence is given.
\begin{lemma}\label{lemma:pathconstant}
    Let $\gamma, \gamma'  \in \Gamma$ . Then, 
    \begin{enumerate}
        \item if $s(\gamma) = I$ and $t(\gamma) = J$, the image
                 \[
                \varphi(\gamma) \in \Ext_G^{\delta(I,J)} (\sigma_I, \sigma_J) 
                \]
   generates the one-dimensional vector space, over $\CC$;
    \item Furthermore, if $\ell(\gamma) = \ell(\gamma')$, $s(\gamma) = s(\gamma')$ and $t(\gamma) = t(\gamma')$, then,
    \[
\gamma - \gamma '  \in \ker(\varphi). 
    \] 
    \end{enumerate}
\end{lemma}

\begin{proof}
    (1): For any $\gamma \in \Gamma$ , we have $\ell(\gamma) = \delta(I,J)$, thus $\gamma$ is made up of arrows that appear in $\gamma$ with multiplicity one. 
    Applying applying Proposition~\ref{prop:dat} inductively on each arrow found in the path, it follows that $\varphi(\gamma) \in \Ext_G^{\delta(I,J)} (\sigma_I, \sigma_J) $. 
    Since this vector space is one-dimensional, the set containing $\varphi(\gamma)$ can be taken to be a basis for the space.\\
    (2): If $\ell(\gamma) = \ell(\gamma')$, $s(\gamma) = s(\gamma')$ and $t(\gamma) = t(\gamma')$, then it follows from (1) that 
    \[
\varphi(\gamma), \varphi(\gamma') \in \Ext_G^{\delta(I,J)}(\sigma_I, \sigma_J),
    \]
    and both are nonzero. 
    Since this space is one-dimensional, there must exist a nonzero constant $c_{\gamma \gamma'}$ so that $\varphi(\gamma) = c_{\gamma \gamma'} \varphi(\gamma')$. Equivalently, 
    \[
\gamma - c_{\gamma \gamma'} \gamma' \in \ker(\varphi). 
    \] 
    since $\varphi$ is a morphism of algebras. However, because we chose each $\varphi(\gamma_{IJ})$ so that they correspond to factors in the same degree $d$ sequence, the choices are compatible with the multiplication in $\Ext_G^\bullet(\Sigma, \Sigma)$ and it follows that $c_{\gamma \gamma'} = 1$. 
\end{proof}

This result implies that every element $\beta \in \Ext_G^\bullet(\Sigma, \Sigma)$ is given by 
\[
\beta = \sum_i a_i \varphi(\gamma_i),
\]
for $a_i \in \CC$ and $\gamma_i \in \Gamma$, forcing all loops in $\Gamma$ to be zero in $\C[Q]/\ker(\varphi)$ and all paths of the same length to be associated up to a constant, whenever they have the same source and target. We can make this precise by describing the kernel of this homomorphism of algebras:

\begin{proposition}\label{prop:quiverkernel}
 The kernel of the surjection $\varphi$ is generated, as a vector space in the following way: 
 \[
\ker(\varphi) = \left\langle\Gamma_0 \cup \Gamma_1 \right\rangle 
 \]
where
\[
\Gamma_0 :=  \{ \gamma- \gamma' \mid s(\gamma) = s(\gamma'),\ t(\gamma) = t(\gamma'),\ \ell(\gamma) = \ell(\gamma') = 2 \}
\]
and 
\[
\Gamma_1 := \{ \gamma \mid s(\gamma) = t(\gamma),\ \ell(\gamma) = 2 \}.
\]
 \end{proposition}
 
 \begin{proof} 
 By Lemma~\ref{lemma:pathconstant}, two paths in $Q$ map to the same non-trivial element in $\Ext_G^\bullet(\Sigma_\lambda, \Sigma_\lambda)$ when they have the same source and target.
 Since any path can be factored into paths of length $1$, it follows by induction on path length that the image of paths in $\Ext_G^\bullet(\Sigma_\lambda, \Sigma_\lambda)$ is uniquely determined by the image of paths of length $2$; this shows that every member of $\Gamma_0$ is in $\ker(\varphi)$. Also, notice that any cycle, and, thus, any member of $\Gamma_1$, is contained in in $\ker(\varphi)$, since 
\[
 \Ext_G^\bullet(\sigma_I, \sigma_I) = \Hom_G(\sigma_I, \sigma_I).
 \]
 Now, suppose that $\gamma = \gamma_m \dots \gamma_1$ is a path of length $m = \delta(s(\gamma), t(\gamma))$, and $\gamma'$ is a path of length $1$ with $t(\gamma') = s(\gamma)$ but $\delta(s(\gamma'\gamma), t(\gamma'\gamma)) \neq \ell(\gamma'\gamma)$. Then, there exists paths $\gamma_m'$ and $ \gamma_{m-1}'$ of length $1$ so that the following conditions hold: 
     \begin{enumerate}
        \item $s(\gamma_m') = t(\gamma_{m-1}') = t(\gamma')$ 
         \item $\gamma = \gamma_m \gamma_{m-1} \gamma_{m-2} \dots \gamma_1 = \gamma'_m \gamma'_{m-1} \gamma_{m-2} \dots \gamma_1 $
     \end{enumerate}
     Then, $\gamma'\gamma_m'$ is a loop of length $2$ and hence, in $\CC[Q]/\ker(\varphi)$, we have
     \begin{align*}
         \gamma' \gamma & = \gamma' (\gamma_m \gamma_{m-1} \gamma_{m-2} \dots \gamma_1 )\\
         & = \gamma'( \gamma_m' \gamma_{m-1}' \gamma_{m-2} \dots \gamma_1) \\
         & = (\gamma'\gamma_m')(\gamma_{m-1}' \gamma_{m-2} \dots \gamma_1 ) \\
         & = 0 
     \end{align*}
     implying that any element of $\text{Path}(Q)\setminus \Gamma$ appears in $\ker(\varphi)$. Moreover, Lemma~\ref{lemma:pathconstant} ensures that $\varphi(\gamma) \neq 0$ for any path in $\gamma \in \Gamma$, providing the result.
 \end{proof}

\section{The Main result}\label{section:main result}
Let $G$ be a split, semisimple connected algebraic group over a $p$-adic field $F$.
Let $\lambda : W_F \to \dualgroup{G}$ be the Langlands parameter for the trivial representation of $G(F)$ and let $X_\lambda$ be the variety of Langlands parameters with this infinitesimal parameter, as in Section~\ref{ssec:Vogan variety}. 
Let $\Sigma_\lambda$ be the direct sum of generalised Steinberg representations of $G(F)$, as in Section~\ref{section:spectral cat}. We let $S$ be the simple roots for $G(F)$. Until now, we have differentiated between $S$ and $\dualgroup{S}$, the set of simple roots for $\dualgroup{G}$. However, as $|S| = |\dualgroup{S}| < \infty $, we fix an order preserving bijection between these sets we will now work only with $S$, since the construction of the quivers at hand depend only on the power set lattice of $S$ (resp. $\dualgroup{S}$), and hence the cardinality of these sets. 
Finally, recall the definition of the principal block $\Perv_{\dualgroup{G}}^\circ(X_\lambda)$ from Section~\ref{subsection:principal block}. 
\begin{lemma}\label{lemma:surjection2} 
Recall the endomorphism algebra $E_\lambda$ given by Definition~\ref{def:Alambda}. There exists a surjection of algebras
\[
\psi:\CC[Q] \twoheadrightarrow E_\lambda 
\]
so that 
\[
\ker(\psi) = \ker(\varphi),
\]
where $\varphi$ is the homomorphism of algebras defined in Lemma~\ref{lemma:quiversurjection}. 
\end{lemma}
\begin{proof}
For any vertex $I$ in the quiver $Q$, let $\psi(\varepsilon_I)\in \Hom_{\Perv_{\Gmult}(\mathbb{A}^1)^{\boxtimes d}}(\mathcal{P}_I, \mathcal{P}_I)$ be the identity map. For any arrow $\gamma_{IJ}: I \rightarrow J$, set 
\[\label{eqn:multbasis}
\psi(\gamma_{IJ}) = f : \mathcal{P}_I \rightarrow \mathcal{P}_J 
\]
where $f$ is equivalent to a map of the form
\[
\tilde{f} = f_1 \otimes f_2 \otimes \dots \otimes f_d,
\]
as was described in Definition~\ref{definition:grading} where $f_i$ is an identity map, on either $\Delta$ or $\nabla$, except in the case where $i$ corresponds to the root found in the (in this case) singleton set $I \ominus J$. Checking the compatibility with the multiplication in $E_\lambda$, apply the universal property of path algebras (again, see \cite{ASS}*{Theorem 1.8}) and lift $\psi$ to the surjection of algebras
\[
\psi:\CC[Q] \twoheadrightarrow E_\lambda. 
\]
As we saw in Corollary \ref{cor:firstequivalence}
\[
\Perv_{H_\lambda}^\circ(V_\lambda)\simeq \Mod(E_\lambda) 
\]
Furthermore, it is easy to see that $E_\lambda$ and $E^{\otimes d}$ are isomorphic, where $E$ is given in Definition~\ref{definition:nabla}, and not simply Morita equivalent. The relations for the base case are stated in Lemma~\ref{lemma:quiver}, and we can therefore use the characterization of relations for the tensor product of path algebras stated in \cite{ZL}*{Lemma 1.3} to write down the relations characterising $\ker(\psi)$ and see that $ \ker(\psi) = \ker(\varphi).$
\end{proof} 
We are now in a position to state our first main result, which establishes the desired equivalence between the extension algebra of generalised Steinberg representations and the principal block of equivariant perverse sheaves on the corresponding Vogan variety.

\begin{theorem}\label{thm:main1}
The category of modules over the extension algebra $\Ext_G^\bullet(\Sigma_\lambda, \Sigma_\lambda)$ is equivalent to the principal block $\Perv_{\dualgroup{G}}^\circ(X_\lambda)$ of the category of equivariant perverse sheaves:
    \[
\Mod(\Ext_{G}^\bullet(\Sigma_\lambda, \Sigma_\lambda))
\simeq 
\Perv_{\dualgroup{G}}^\circ(X_\lambda) . 
    \]
\end{theorem}

\begin{proof}
Recall the endomorphism algebra $E_\lambda$ from Definition~\ref{def:Alambda}. 
Lemma~\ref{lemma:quiversurjection} and Lemma~\ref{lemma:surjection2} give isomorphisms 
\[
\Ext_G^\bullet(\Sigma_\lambda, \Sigma_\lambda) \cong \CC[Q]/\ker(\varphi) =  \CC[Q]/\ker(\psi) \cong E_\lambda .
\]
By Proposition~\ref{proposition:endomorphism}, we have an equivalence of categories 
  \[
\Perv_{H_\lambda}^\circ(V_\lambda) \simeq \Mod(E_\lambda). 
\]
This implies the result. 
\end{proof}

\begin{example}
Returning to Example~\ref{example:1} from the end of Section~\ref{ssec:SimpleObjectsPrincipalBlock}, set $G = \text{PGL}(2)$ and let $\lambda$ be the unramified infinitesimal parameter given by $\lambda(\Frob) = \diag(q^{1/2}, q^{-1/2})$.
Now can visualise the equivalence appearing in Theorem~\ref{thm:main1} explicitly: in this case $\Perv_{H_\lambda}^\circ(V_\lambda)$ is equivalent to the category of finite dimensional modules over the path algebra from the base case, as we saw in Lemma~\ref{lemma:quiver}. From the stratification of $V_\lambda \cong \mathbb{A}^1$ into $C_0 \cong \{0\}$ and $C_1 \cong \Gmult$, the Langlands correspondence gives
\[
\1_{\text{PGL}_2(F)} \leftrightarrow \IC(\1_{C_0}); \qquad \St_{\text{PGL}_2(F)} \leftrightarrow \IC(\1_{C_1})
\]
Since $\text{PGL}(2)$ has a unique simple root, the description outlined in describes only two short exact sequences between these representations: 
\[\begin{tikzcd}[ampersand replacement=\&]
	0 \& {\1_{\text{PGL}_{2}(F)}} \& {i_B^G\left(\delta^{-1/2}\right)} \& {\St_{\text{PGL}_2(F)}}\& 0 \\
	0 \& {\St_{\text{PGL}_2(F)}} \& {i_{\overline{B}}^G\left(\delta^{-1/2}\right)} \& {\1_{\text{PGL}_2(F)} } \& 0
	\arrow[from=1-1, to=1-2]
	\arrow[from=1-2, to=1-3]
	\arrow[from=1-3, to=1-4]
	\arrow[from=1-4, to=1-5]
	\arrow[from=2-1, to=2-2]
	\arrow[from=2-2, to=2-3]
	\arrow[from=2-3, to=2-4]
	\arrow[from=2-4, to=2-5]
\end{tikzcd}\]
Where $B \subset G$ is the chosen Borel, and $\delta$ is the modulus character for the choice of maximal torus. Here, $\Sigma_\lambda = \1_{\text{PGL}_{2}(F)} \oplus \St_{\text{PGL}_{2}(F)}$, and the elements corresponding to the pair of extensions above must compose to zero in the algebra, since there are no higher order extension. This provides the isomorphism with $E_\lambda$, and thus the equivalnece of categories from Theorem~\ref{thm:main1}, for the base case. 
\end{example}

This equivalence mirrors a remarkable degree of symmetry between either side of the Langlands correspondence. For instance, consider the behaviour of Aubert duality on generalised Steinberg representations (resp. the Fourier transform on $\Perv_{\dualgroup{G}}^\circ(X_\lambda)$): as we say in Lemma~\ref{lem:aubertdualofsteinberg}, the Aubert-Zelevinksy dual of generalised representations is also governed by subsets of simple roots, as they satisfy the equation $\AZ(\sigma_I) \cong \sigma_{I^c}$. In the same way, it is well known that (associating $\mathbb{A}^1$ with its dual) we have $\Ft(\IC(\1_{\Gmult})) \cong \IC(\1_0)$ and, since the the functor is exact, it interchanges the projective cover for these objects, so
\[
\Ft(\Delta) \cong \nabla. 
\]
from which it follows that 
\[
\Ft(\mathcal{P}_{I}) \cong \mathcal{P}_{I^c}. 
\]
in the decomposition $\Perv_{\Gmult}(\mathbb{A}^1)^{\boxtimes d}$, since the Fourier transform is well known . Thus, these functors induce equivalent automorphisms of $\Ext_G^\bullet(\Sigma_\lambda, \Sigma_\lambda)$ and $E_\lambda$, ensuring a compatibility with the equivalence from Theorem~\ref{thm:main1} so that the diagram 
\[\begin{tikzcd}[ampersand replacement=\&]
	{\Mod(\Ext_G^\bullet(\Sigma_\lambda, \Sigma_\lambda) )} \&\& {\Perv_{\dualgroup{G}}^\circ(X_\lambda)} \\
	\\
	{\Mod(\Ext_G^\bullet(\Sigma_\lambda, \Sigma_\lambda) )} \&\& {\Perv_{\dualgroup{G}}^\circ(X_\lambda)}
	\arrow["\sim", from=1-1, to=1-3]
	\arrow["\AZ"', from=1-1, to=3-1]
	\arrow["\Ft", from=1-3, to=3-3]
	\arrow["\sim", from=3-1, to=3-3]
\end{tikzcd}\]
commutes, up to an equivalence of categories. 

In fact, since $\Delta$ and $\nabla$ are both projective and injective in $\Perv_{\Gmult}(\mathbb{A}^d)$, it follows immediately from Lemma~\ref{lemma:Bourbaki} (after realising that the case for injective objects of tensor products of finite dimensional algebras is exactly the dual statement of the second part of this lemma) that the projective and injective objects coincide in the category of $E_\lambda$-modules implying that $E_\lambda$, and hence $\Ext_G^\bullet(\Sigma_\lambda, \Sigma_\lambda)$, is a self-injective algebra (meaning that is injective object over itself, which follows since the projective decomposition of the regular representation must likewise be a decomposition into injective objects). From this point of view, the functor induced by Aubert-Zelevinksy duality is easily shown to be none other than the Nakayama functor
$$
D\circ \Hom%_{\Ext_G^\bullet(\Sigma_\lambda, \Sigma_\lambda)}
( - , \Ext_G^\bullet(\Sigma_\lambda, \Sigma_\lambda)) : \Mod(\Ext_G^\bullet(\Sigma_\lambda, \Sigma_\lambda)) \xrightarrow{\sim} \Mod(\Ext_G^\bullet(\Sigma_\lambda, \Sigma_\lambda)), 
$$
where $D = \Hom_{\CC}(-,\CC)$ is used to denote the usual duality on finite dimensional algebras, which may be viewed here as being in correspondence with contragradience for representations of $p$-adic groups; see \cite{ASS}*{III.2} for a discussion of the properties of this functor for arbitrary finite dimensional algebras. 
Indeed, in the case of a Koszul, self-injective algebra $A$, the Nakayama functor produces an analogus duality on a certain derived category of $A^!$ modules, given by $D\Ext_{A^!}^d(- , A^!)$, where $d$ is the graded global dimension of $A^!$ \cite{RMV99}*{Theorem 5.6}. Of course, this duality immediately recalls the definition of Aubert-Zelevinsky duality discussed throughout Section~\ref{ssection:AZ}, indicating that the Hecke algebra $\mathcal{H}(G)$ may be replaced by $\Ext_G^\bullet(\Sigma_\lambda, \Sigma_\lambda)^!$ in the definition of this duality on a certain subcategory of $\Rep(G)$, and we close our paper with a realisation of this Koszul dual via the $\dualgroup{G}$-equivariant cohomology of $X_\lambda$. 
\section{Equivariant extensions of perverse sheaves} 
In this last section, we complete the picture by demonstrating the complementary result that the extension algebra of generalised Steinberg representations is Koszul dual to the extension algebra of simple, $\dualgroup{G}$-equivariant intersection cohomology complexes on $X_\lambda$, being the endomorphism algebra of these complexes in the $\dualgroup{G}$-equivariant derived category of constructible sheaves on $X_\lambda$. To do this, we use a remarkable result from \cite{Braden-Lunts} that the equivariant derived category for a toric variety is equivalent to the derived category of its perverse heart, giving an equivariant analog of Beilinson's theorem which is well known \emph{not} to hold in general. This equivalence allows us to consider morphisms in $D^b_{\dualgroup{G}}(X_\lambda)$ as extensions of $E_\lambda$-modules and use Theorem~\ref{thm:main1} to produce a Koszul duality between extension algebras on either side of the Langlands correspondence.

To begin, let us recall the construction of the equivariant derived category $D^b_{\dualgroup{G}}(X_\lambda)$. There are several different descriptions of this category, all of which turn out to be equivalent \cite{Geoff}*{Theorem 6.0.1}. Naively, this is the bounded, derived category on the space $X_\lambda/\dualgroup{G}$. Of course, this space is often not a variety, but a stack, in which case it is often convenient to pass to one of the other, equivalent conceptions of this category. Following \cite{WW}*{Definition 2}, we give a definition for the equivariant derived category that will be more amenable to calculation in our context: 
\begin{definition}
    Let $V$ be an $H$-variety and let $EH$ be a free, contractable $H$-space, giving the usual quotient maps 
    \[
X \overset{p}\twoheadleftarrow EH \times V \overset{q}\twoheadrightarrow EH \times^H V. 
    \]
    Then, the $\textit{bounded equivariant derived category}$ on $V$ is given by the full subcategory of $D^b(EH \times^H V)$ given by the objects $\mathcal{F}$ satisfying $q^* \mathcal{F} \cong p^*\mathcal{G}$, 
    for some constructible complex $\mathcal{H}$ in $D^b_c(V)$. 
\end{definition}
This description is equivalent to the more familiar construction from \cite{BL}, as is outlined in \cite{WW}*{Remark 1}. The category of equivariant perverse $\Perv_H(V)$ can thus be described as the perverse heart of either of these categories. 
\begin{lemma}\label{lem:freeequivalence}
    Let $V$ be a toric variety, given by the action of a torus $H$, and let $A$ be an algebra whose category of finite-dimensional modules is equivalent to $\Perv_H(V)$. Then, there is a equivalence of triangulated categories
    \[
D_H^b(V) \simeq D^b(A). 
    \]
\end{lemma}
\begin{proof}
    Following \cite{Braden-Lunts}*{Corollary 7.3.3}, there is an equivalence of categories
    \begin{align*}
        D^b(\Perv_H(V)) \simeq D^b_H(V)
    \end{align*}
    The assumed equivalence $\Perv_G(X) \simeq \textbf{Mod}(A)$ then induces an equivalence of their corresponding bounded derived categories. 
\end{proof}

We are interested in the extension algebra of simple equivariant perverse sheaves in the principal block of $\Perv_{H_\lambda}^\circ(V_\lambda)$. By the result above, this should be equivalent to considering the endomorphism algebra of this object in the equivariant derived category, and it would be convenient if we could consider it in the simpler category $D^b_{\Gmult^d}(V_\lambda)$, given by component-wise multiplication of the torus on $V_\lambda$. To demonstrate this combinatorial convenience, it suffices to show that the embedding 
\[
\Perv_{\Gmult^d}(V_\lambda) \simeq \Perv_{H_\lambda}^\circ(V_\lambda) \hookrightarrow \Perv_{H_\lambda}(V_\lambda)
\]
is closed under extensions. 
\begin{lemma}
    For any $I, J \subseteq S$, the intersection of the closures of $C_I, C_J \subset V_\lambda$ is given by 
    \[
\overline{C}_I \cap \overline{C}_J \cong \overline{C}_{I \cup J}
    \]
\end{lemma}

\begin{proof}
    Recall from Section~\ref{ssec:Vogan variety} that any orbit $C_I \subset V$. The closure of this orbit is given by
    \[
\overline{C_I} = \{ (x_1, \dots, x_d) \mid x_i = 0 \text{ if } i \in I\} 
    \]
    Hence,
    \[
\overline{C}_I \cap \overline{C}_J = \{(x_1, \dots, x_d) \mid x_i = 0 \text{ if } i \in I \text{ or } i \in J \} 
\]
which is exactly the description of $\overline{C}_{I \cup J}$. 
\end{proof}
\begin{lemma}\label{lem:PointCohomology}
    Let $V$ be a contractible $H$-space. Then, 
    \[
H^\bullet_H(V) \cong H^\bullet_H(\text{pt})
    \]
\end{lemma}
\begin{proof}
    The usual projection $EH \times V \rightarrow EH$ gives a fibration
    \[
EH \times^H V := (EH \times V)/H \rightarrow EH/H =: BH. 
    \]
    This fibration has contractible fibers, since $V$ is contractible by assumption, and hence the map is a homotopy equivalence. 
\end{proof}
\begin{proposition}\label{prop:CohomologyAndExtensions}
    Let $I, J \subseteq S$. Then, there are isomorphisms of vector spaces
    \begin{align*}
 (i). \quad &\Ext_{H_\lambda}^\bullet(\IC(\mathbbm{1}_I), \IC(\mathbbm{1}_J)) \cong H^\bullet_{H_\lambda} (\overline{C}_{I \cup J} ) [\delta(I,J)]  \\
(ii). \quad & \Ext_{\Gmult^d}^\bullet(\IC(\mathbbm{1}_I), \IC(\mathbbm{1}_J)) \cong H^\bullet_{\Gmult^d} (\overline{C}_{I \cup J} ) [\delta(I,J)],  
    \end{align*}
where the notation in the latter half of each isomorphism is meant to denote the appropriate cohomology space of the orbit closure $\overline{C}_{I\cup J}$, shifted by $\delta(I,J)$.
\end{proposition}
\begin{proof}
    For $(i)$, recall from Section~\ref{ssec:Vogan variety} recall that every $H_\lambda$-orbit of $V_\lambda$ can be broken into component parts where every component is given by either $\Gmult$ or $\{0\}$. In particular, every orbit is smooth and hence 
    \[
\IC_I \cong {}^Ii_* \mathbbm{1}_{\overline{C}_I}\,[\dim C_I], 
    \]
where $\,^Ii : \overline{C_I} \to V_\lambda$ is inclusion.
    Hence, letting $d_I = \dim C_I$ and $d_J = \dim C_J$, we can use adjunction pushing to calculate
    \begin{align*}
        \Ext_{H_\lambda}^d(\IC_I, \IC_J) & \cong \Hom_{D^b_{H_\lambda}(V_\lambda)}({}^Ii_* \mathbbm{1}_{\overline{C}_I}[d_I], {}^Ji_* \mathbbm{1}_{\overline{C}_J}[d_J + d] ) \\
        & \cong \Hom_{D^b_{H_\lambda}(\overline{C}_J)}({}^Ji^*\ {}^Ii_* \mathbbm{1}_{\overline{C}_I}[d_I],  \mathbbm{1}_{\overline{C}_J}[d_J + d] ) \\
        & \cong \Hom_{D^b_{H_\lambda}(\overline{C}_{I \cup J})}( \mathbbm{1}_{\overline{C}_{I\cup J}}[d_I], {}^{I\cup J}i^!\ {}^Ji_* \mathbbm{1}_{\overline{C}_J}[d_J + d] ). \\
    \end{align*}
    Now, from the interaction of Verdier duality with six-functor formalism (\emph{c.f.}, \cite{Achar}*{\S 2.8}), it follows that
    \begin{align*}
\mathbb{D} {}^{I \cup J} i^! \mathbbm{1}_{\overline{C}_J}[d_J + d] & \cong {}^{I \cup J}i^* \mathbbm{D} \mathbbm{1}_{\overline{C}_J}[ d_J + d] \\
& = \mathbbm{1}_{\overline{C}_{I\cup J}}[ 2d_J - d_J - d]\\
& = \mathbbm{1}_{\overline{C}_{I\cup J}}[ d_J - d]
\end{align*} 
Hence, applying Verdier duality to both sides of the homset above, we obtain 
\begin{align*}
    \Hom_{D_{H_\lambda}^b(\overline{C}_{I \cup J})}(\mathbbm{1}_{\overline{C}_{I\cup J}}[d_I] &, {}^{I\cup J}i^!\ {}^Ji_* \mathbbm{1}_{\overline{C}_J}[d_J + d] )
    \\
    & \cong  \Hom_{D_{H_\lambda}^b(\overline{C}_{I \cup J})}(\mathbb{D} {}^{I\cup J}i^!\ {}^Ji_* \mathbbm{1}_{\overline{C}_J}[d_J + d], \mathbb{D} \mathbbm{1}_{\overline{C}_{I\cup J}}[d_I])
    \\
    & \cong  \Hom_{D_{H_\lambda}^b(\overline{C}_{I \cup J})} ( \mathbbm{1}_{\overline{C}_{I \cup J} } [ d_J - d],\mathbbm{1}_{\overline{C}_{I\cup J}}[2 d_{I \cup J} - d_I])\\
    & \cong H^{\bullet}_{H_\lambda}(\overline{C}_{I \cup J}) [d - d_I + d_J - 2d_{I \cup J}]
\end{align*}
Thus, the vector space is non-zero if and only if $d_J - d = 2d_{I \cup J} - d_I$ or, equivalently, $d = d_I + d_J - 2d_{I \cup J}$, in which case it is one-dimensional. Since $d_I = \abs{S \setminus I} = \abs{I^c}$ for all $I \subseteq S$, it is easy to see that 
\[
d = d_I + d_J - 2d_{I \cup J} = \delta(I,J) 
\]
providing the result. The result for $(ii)$ follows, \textit{mutatis mutandis}. 
\end{proof}

\begin{proposition}\label{prop:algebracomparison}
    There is an isomorphism of algebras 
    \[
\Ext_{H_\lambda}^\bullet ( \mathcal{T}_\lambda , \mathcal{T}_\lambda ) \cong \Ext_{\Gmult^d}^\bullet ( \mathcal{T}_\lambda , \mathcal{T}_\lambda ) 
    \]
    where $\mathcal{T}_\lambda$ is the direct sum of all intersection cohomology complexes of the form $\IC(\mathbbm{1}_{C})$ for some orbit $C \subset V_\lambda$, viewed as an object in both $D_{H_\lambda}^b(V_\lambda)$ and $D_{\Gmult^d}^b(V_\lambda)$. 
\end{proposition}
\begin{proof}
    In this context, the closure of every orbit $C \subset V_\lambda$ is a vector space, and thus contractible. Furthermore, by Proposition~\ref{prop:CohomologyAndExtensions}, for a pair of subsets of simple roots $I,J \subseteq S$, there are isomorphisms 
    \[
\Ext_{H_\lambda}^\bullet(\mathcal{IC}_I, \mathcal{IC}_J) \cong H_{H_\lambda}^\bullet(\overline{C}_{I \cup J})[\delta(I,J)], \quad   \Ext_{\Gmult^d}^\bullet(\mathcal{IC}_I, \mathcal{IC}_J) \cong H_{\Gmult^d}^\bullet(\overline{C}_{I \cup J})[\delta(I,J)].  
    \]
This space is contractible, and so its equivariant cohomology is isomorphic to the equivariant cohomology of a point, as was shown in Lemma \ref{lem:PointCohomology}, which is to say that
\[
H_{H_\lambda}^\bullet(\overline{C}_{I \cup J}) \cong H^\bullet(B H_\lambda), \quad H_{\Gmult^d}^\bullet(\overline{C}_{I \cup J}) \cong H^\bullet(B \Gmult^d). 
\]
The map $H_\lambda \twoheadrightarrow \Gmult^d$ from Proposition~\ref{prop:embedding} is an isogeny, hence the classifying space of the two groups shares the same complex homotopy type. Thus, 
\[
H^\bullet(B H_\lambda) \cong H^\bullet(B \Gmult^d), 
\]
and we are done.
\end{proof}
This means that the extension algebra for the direct sum of simple perverse sheaves of the form $\IC(\1_C)$ has equivalent descriptions in both $D_{H_\lambda}^b(V_\lambda)$ and $D_{\Gmult^d}^b(V_\lambda)$, which is to say that no $H_\lambda$-equivariant extensions between any two perverse sheaves of this form factor through some other simple intersection cohomology complex found in $\Perv_{H_\lambda}(V_\lambda)$, \emph{i.e.}, an intersection cohomology complex that corresponds to a representation of a pure inner form of $G$, instead of $G$ itself, under the Langlands correspondence. We can then use the equivalence 
\[
D_{H_\lambda}^\bullet(V_\lambda) \simeq D_{\dualgroup{G}}^\bullet(X_\lambda) 
\]
to get our final result: 
\begin{theorem}\label{thm:main2}
    There is an isomorphism of algebras 
    \[
\Ext_G^\bullet (\Sigma_\lambda, \Sigma_\lambda)^! \simeq \Ext_{\dualgroup{G}}^\bullet(\mathcal{S}_\lambda , \mathcal{S}_\lambda), 
    \]
    where $\Ext_G^\bullet(\Sigma_\lambda, \Sigma_\lambda)^! $ denotes the Koszul dual of the extension algebra of generalised Steinberg representations of $G$. 
\end{theorem}
\begin{proof}
    We can produce an isomorphism of algebras 
    \[
\Ext_{\dualgroup{G}}^\bullet(\mathcal{S}_\lambda, \mathcal{S}_\lambda) \cong \Ext^\bullet_{H_\lambda}(\mathcal{T}_\lambda, \mathcal{T}_\lambda) 
    \]
    using the equivalence of triangulated categories provided by the induction equivalence along inclusion of algebraic groups $H_\lambda \subset \dualgroup{G}$, where we choose the equivalence so that $\mathcal{T}_\lambda$ is the same perverse sheaf from Proposition~\ref{prop:algebracomparison}. By the conclusion of this proposition, it suffices to show that $\Ext_{\Gmult^d}^\bullet(\mathcal{T}_\lambda , \mathcal{T}_\lambda)$ is the appropriate Koszul dual. This is the endomorphism algebra of $\mathcal{T}_\lambda$ in the equivariant derived category $D_{\Gmult^d}^b(V_\lambda)$. By \cite{Braden-Lunts}*{Corollary 7.3.3},
    \[
D^b_{\Gmult^d}(\Perv(V_\lambda)) \simeq D_{\Gmult^d}^b(V_\lambda).
    \]
    As we saw in the proof of Theorem \ref{thm:main1}, there is an equivalence of categories
    \[
F: \Perv_{\Gmult^d}(V_\lambda) \xrightarrow{\sim} \Mod(E_\lambda). 
    \]
    For this reason, it follows from Lemma \ref{lem:freeequivalence} that the calculation of extensions in $D^b_{\Gmult}(V_\lambda)$ is equivalent to their calculation in $D^b(E_\lambda)$. Hence,  
    \[
\Ext_{\Gmult^d}^\bullet (\mathcal{T}_\lambda, \mathcal{T}_\lambda) \cong \Ext_{E_\lambda}^\bullet (F(\mathcal{T}_\lambda), F(\mathcal{T}_\lambda)), 
    \] 
    the latter being $E_\lambda^!$, and we saw that $E_\lambda$ is Koszul itself in Section~\ref{ssec:koszulalgebra}. As Theorem \ref{thm:main1} shows, there is an isomorphism
    \[
E_\lambda \cong \Ext_G^\bullet(\Sigma_\lambda, \Sigma_\lambda),   
    \]
   implying an isomorphism between their respective Koszul duals.
\end{proof}
This final result, Theorem~\ref{thm:main2}. can be viewed as a $p$-adic analogue of Soergel's conjecture for real groups. Generally, Soergel's philosophy is that the local Langland's correspondence should be interpreted as a Koszul duality between perverse sheaves on Vogan varieties\footnote{This is often referred to as the ``ABV perspective" when treating Archimedean groups, or groups over local fields more generally, given \cite{ABV}.} built from data on the Galois side of the Langlands correspondence and their corresponding representations on the automorphic side of the correspondence, as Soergel so eloquently laid out with Beilinson and Ginzburg in their landmark paper \cite{BGS}, providing the case for the principal block of the category $\mathcal{O}$. 

Presently, we have laid out an incarnation of Soergel's philosophy for $p$-adic groups, in the case of generalised Steinberg representations, with Theorem~\ref{thm:main1}, and Theorem~\ref{thm:main2} demonstrates that this is a true Koszul duality of algebras built from data on either side of the correspondence. We expect an analogous statement to be true for a much larger class of representations, and for a much larger class of groups, though with necessary modifications to the cohomology theory for representations of $p$-adic groups (for instance, as is noted by \cite{Orlik}*{Corollary 18}, generalised Steinberg representations for $\GL_n(F)$ enjoy an additional self-extension that has no corresponding structure in the geometry of the Vogan variety).

\end{document}